\documentclass[final,12pt]{elsarticle}
\usepackage{srcltx}
\usepackage{float}
\usepackage{eurosym}
\usepackage{mathtools}
\usepackage{amsmath}
\usepackage{cases}
\usepackage{amsfonts}
\usepackage{amssymb}
\usepackage{amsthm}
\usepackage{graphicx}
\usepackage{mathrsfs}
\usepackage{xcolor}
\usepackage{xpatch}
\usepackage{exscale}
\usepackage{latexsym}
\usepackage{tikz}
\usepackage{caption}
\usepackage{lineno}
\xpatchcmd{\MaketitleBox}{\hrule}{}{}{}
\xpatchcmd{\MaketitleBox}{\hrule}{}{}{}
\usepackage[colorlinks,plainpages=true,pdfpagelabels,hypertexnames=true,colorlinks=true,pdfstartview=FitV,linkcolor=blue,citecolor=red,urlcolor=black]{hyperref}
\PassOptionsToPackage{unicode}{hyperref}
\PassOptionsToPackage{naturalnames}{hyperref}
\usepackage{enumerate}
\usepackage[shortlabels]{enumitem}
\usepackage{bookmark}
\usepackage{wasysym}
\usepackage{esint}
\usepackage[ddmmyyyy]{datetime}
\usepackage[margin=2cm]{geometry}
\makeatletter
\g@addto@macro\normalsize{%
  \setlength\abovedisplayskip{4pt}
  \setlength\belowdisplayskip{4pt}
  \setlength\abovedisplayshortskip{4pt}
  \setlength\belowdisplayshortskip{4pt}
}
\numberwithin{equation}{section}
\everymath{\displaystyle}
\usepackage[capitalize,nameinlink]{cleveref}
\crefname{section}{Section}{Sections}
\crefname{subsection}{Subsection}{Subsections}
\crefname{condition}{Condition}{Conditions}
\crefname{hypothesis}{Hypothesis}{Conditions}
\crefname{assumption}{Assumption}{Assumptions}
\crefname{lemma}{Lemma}{Lemmas}
\crefname{definition}{Definition}{Definitions}

\crefformat{equation}{\textup{#2(#1)#3}}
\crefrangeformat{equation}{\textup{#3(#1)#4--#5(#2)#6}}
\crefmultiformat{equation}{\textup{#2(#1)#3}}{ and \textup{#2(#1)#3}}
{, \textup{#2(#1)#3}}{, and \textup{#2(#1)#3}}
\crefrangemultiformat{equation}{\textup{#3(#1)#4--#5(#2)#6}}%
{ and \textup{#3(#1)#4--#5(#2)#6}}{, \textup{#3(#1)#4--#5(#2)#6}}%
{, and \textup{#3(#1)#4--#5(#2)#6}}

\Crefformat{equation}{#2Equation~\textup{(#1)}#3}
\Crefrangeformat{equation}{Equations~\textup{#3(#1)#4--#5(#2)#6}}
\Crefmultiformat{equation}{Equations~\textup{#2(#1)#3}}{ and \textup{#2(#1)#3}}
{, \textup{#2(#1)#3}}{, and \textup{#2(#1)#3}}
\Crefrangemultiformat{equation}{Equations~\textup{#3(#1)#4--#5(#2)#6}}%
{ and \textup{#3(#1)#4--#5(#2)#6}}{, \textup{#3(#1)#4--#5(#2)#6}}%
{, and \textup{#3(#1)#4--#5(#2)#6}}

\crefdefaultlabelformat{#2\textup{#1}#3}
\numberwithin{equation}{section}

\newtheorem{theorem} {Theorem}[section]

\newtheorem{lemma}{Lemma}[section]

\newtheorem{example}{Example}[section]
\newtheorem{counter example}{Counter Example}[section]
\newtheorem{remark} {Remark}[section]
\newtheorem{definition} {Definition}[section]

\def\CC{{\rm \kern.24em \vrule width.02em height1.4ex depth-.05ex \kern-.26emC}}

\def\TagOnRight

\def\AA{{it I} \hskip-3pt{\tt A}}

\def\QQ{\rlap {\raise 0.4ex \hbox{$\scriptscriptstyle |$}} {\hskip -0.1em Q}}

\makeatletter
\newcommand{\vo}{\vec{o}\@ifnextchar{^}{\,}{}}
\makeatother

\def\YYint#1#2#3{{\setbox0=\hbox{$#1{#2#3}{\iint}$}
    \vcenter{\hbox{$#2#3$}}\kern-.50\wd0}}

\def\XXint#1#2#3{{\setbox0=\hbox{$#1{#2#3}{\int}$}
    \vcenter{\hbox{$#2#3$}}\kern-.50\wd0}}

\makeatletter
\def\namedlabel#1#2{\begingroup
   \def\@currentlabel{#2}%
   \label{#1}\endgroup
}
\makeatother
\makeatletter
\newcommand{\rmh}[1]{\mathpalette{\raisem@th{#1}}}
\newcommand{\raisem@th}[3]{\hspace*{-1pt}\raisebox{#1}{$#2#3$}}
\makeatother

\newcommand{\descref}[2]{\hyperref[#1]{\textnormal{\textcolor{black}{(}\textcolor{blue}{\bf #2}\textcolor{black}{)}}}}

\newcommand{\dref}[2]{\hyperref[#1]{\textcolor{black}{(}\textcolor{blue}{\bf #2}\textcolor{black}{)}}}
\newcommand{\be} {\begin{eqnarray}}
\newcommand{\ee} {\end{eqnarray}}
\newcommand{\Bea} {\begin{eqnarray*}}
\newcommand{\Eea} {\end{eqnarray*}}
\newcommand{\rr}{\rightarrow}

\newcommand{\la} {\lambda}

\newcommand{\f}{\infty}

\newcommand{\R}{\mathbb{R}}

\newcommand{\lab} {\label}

\newcommand{\sgn}{\mathop\mathrm{sgn}}
\newcommand{\al}{\alpha}

\newcommand{\D}{\Delta}
\newcommand{\Z}{\mathbb{Z}}
\newcommand{\sumj}{\sum_{j \in \Z}}
\newcommand{\BV}{\textrm{BV}}
\newcommand{\mM}{\mathcal{M}}
\newcommand{\jph}{j+1/2}
\newcommand{\jmh}{j-1/2}
\newcommand{\nph}{n+1/2}
\newcommand{\mF}{\mathcal{F}}
\newcommand{\mK}{\mathcal{K}}
\newcommand{\mP}{\mathcal{P}}
\newcommand{\mC}{\mathcal{C}}
\newcommand{\mD}{\mathcal{D}}

\DeclareMathOperator{\loc}{loc}

\DeclareMathOperator{\TV}{TV}

\newcommand{\norm}[1]{\left|\hspace{-0.2mm}\left| #1 \right|\hspace{-0.2mm}\right|}
\newcommand{\abs}[1]{\left| #1\right|}

\newcounter{whitney}
\refstepcounter{whitney}

\newcounter{ineqcounter}
\refstepcounter{ineqcounter}
\makeatletter
\def\ps@pprintTitle{%
\let\@oddhead\@empty
\let\@evenhead\@empty
\def\@oddfoot{}%
\let\@evenfoot\@oddfoot}
\makeatother
\usepackage[doublespacing]{setspace}
\usepackage[titletoc,toc,page]{appendix}

\usepackage{pgf,tikz}
\usetikzlibrary{arrows}
\usetikzlibrary{decorations.pathreplacing}
\begin{document}

\begin{frontmatter}

\title{Convergence of a Godunov scheme for degenerate conservation laws with  BV spatial flux and a study of Panov type fluxes}

\author[myaddress1]{Shyam Sundar Ghoshal}
\ead{ghoshal@tifrbng.res.in}


\address[myaddress1]{Centre for Applicable Mathematics,Tata Institute of Fundamental Research, Post Bag No 6503, Sharadanagar, Bangalore - 560065, India.}

\author[myaddress2]{John D. Towers}
\ead{john.towers@cox.net}

\author[myaddress1]{Ganesh Vaidya}
\ead{ganesh@tifrbng.res.in}

\address[myaddress2]
{MiraCosta College, 3333 Manchester Avenue, Cardiff-by-the-Sea, CA 92007-1516, USA.}

\begin{abstract}
	In this article we prove convergence of the Godunov scheme of \cite{GJT_2019}
	for a scalar conservation law in one space dimension with a
	spatially discontinuous flux. There may be infinitely many flux discontinuities, and the set of discontinuities may have
	accumulation points. Thus the existence of traces cannot be assumed. In contrast to the study appearing in 		\cite{GJT_2019}, we do not restrict the flux to be unimodal. We allow for the case where the flux
 	has degeneracies, i.e., the flux may vanish on some interval
 	of state space. Since the flux is allowed to be degenerate, the corresponding singular map may not be invertible, and thus 	the convergence proof appearing in \cite{GJT_2019} does not
 	pertain. We prove that the Godunov approximations nevertheless do converge in the presence
	of flux degeneracy, using an alternative method of proof.
	We additionally consider the case where the flux has the form described in \cite{Panov2009a}. For this case
	we prove convergence via yet another method.
	This method of proof provides a spatial variation bound on the solutions, which is of independent interest.	
	We present numerical examples that illustrate the theory.
	\end{abstract}

\end{frontmatter}
\tableofcontents
\section{Introduction}\label{section_intro}
In this article we prove convergence of the Godunov scheme of \cite{GJT_2019} to an adapted entropy solution, as
defined in \cite{GTV_2020},
of  the initial value problem for scalar conservation laws with spatially dependent flux given by
\begin{eqnarray}\label{1}
u_t+A(x,u)_x&=& 0\quad\quad\quad \text{for}\,\,\,(t,x) \in(0,T)\times\R=:Q,\label{eq:discont}\\
\label{2}u(x,0)&=& u_0(x),\,\quad\text{for}\,\,\, x\in \mathbb{R}.\label{eq:data}
\end{eqnarray}
Here the set of spatial
discontinuities  of the  flux $A(x,u)$  is allowed to be infinite with accumulation points. In contrast to the study appearing in \cite{GJT_2019}, we do not restrict the flux to be unimodal. We allow for the case where the flux
 $A(x,u)$ has degeneracies, i.e.,  $u \rightarrow A_u(x,u)$ may vanish on some $x$-dependent interval
 of state space $(u_M^-(x),u_M^+(x))$, see Figure~\ref{figure:two_fluxes}. 
 Well-posedness of adapted entropy solutions for the case of a degenerate flux was established in \cite{GTV_2020}.
 Since the flux is allowed to be degenerate, the corresponding singular map may not be invertible, and thus the convergence proof appearing in \cite{GJT_2019} does not
 pertain. We instead use the compactness method of \cite{BGKT} to prove convergence. We also provide another, alternative,
 convergence proof when the flux has the form $A(x,u) = g(\beta(x,u))$ as in \cite{Panov2009a}. This method of proof provides a spatial variation bound on the solutions, which is of independent interest. Spatial variation bounds are not generally available
 for solutions of conservation laws with discontinuous flux (see \cite{ADGG,GTV_2020,Ghoshal-JDE}).

\begin{figure}
	\centering

\tikzset{every picture/.style={line width=0.75pt}} 

\begin{tikzpicture}[x=0.75pt,y=0.75pt,yscale=-1,xscale=1]

\draw [line width=1.5]    (68.5,329.99) -- (599.96,330.06) ;
\draw [shift={(602.96,330.06)}, rotate = 180.01] [color={rgb, 255:red, 0; green, 0; blue, 0 }  ][line width=1.5]    (14.21,-4.28) .. controls (9.04,-1.82) and (4.3,-0.39) .. (0,0) .. controls (4.3,0.39) and (9.04,1.82) .. (14.21,4.28)   ;
\draw [color={rgb, 255:red, 0; green, 0; blue, 200 }  ,draw opacity=1 ][line width=1.5]    (406.58,270.62) -- (462.16,270.62) ;
\draw [color={rgb, 255:red, 0; green, 0; blue, 200 }  ,draw opacity=1 ][line width=1.5]    (197.33,270.62) -- (265.67,270.62) ;
\draw [line width=2.25]  [dash pattern={on 2.53pt off 3.02pt}]  (198.09,269.79) -- (197.78,297.42) -- (197.51,326.72) ;
\draw [line width=2.25]  [dash pattern={on 2.53pt off 3.02pt}]  (407.52,271.46) -- (407.2,299.93) -- (406.93,329.23) ;
\draw [line width=2.25]  [dash pattern={on 2.53pt off 3.02pt}]  (461.13,271.46) -- (460.82,299.09) -- (460.55,328.39) ;
\draw [color={rgb, 255:red, 0; green, 0; blue, 200 }  ,draw opacity=1 ][line width=1.5]    (265.56,270.61) .. controls (301.57,241.98) and (283.07,225.94) .. (331.28,172.24) ;
\draw [color={rgb, 255:red, 0; green, 0; blue, 200 }  ,draw opacity=1 ][line width=1.5]    (331.28,172.24) .. controls (377.08,127.89) and (343.65,117.46) .. (391.85,63.76) ;
\draw [color={rgb, 255:red, 0; green, 0; blue, 200 }  ,draw opacity=1 ][line width=1.5]    (197.39,269.78) .. controls (158.23,232.24) and (182.73,225.41) .. (138.02,170.23) ;
\draw [color={rgb, 255:red, 0; green, 0; blue, 200 }  ,draw opacity=1 ][line width=1.5]    (138.02,170.23) .. controls (95.12,124.25) and (129.13,116.43) .. (84.42,61.25) ;
\draw [line width=2.25]  [dash pattern={on 2.53pt off 3.02pt}]  (265.95,270.62) -- (265.63,298.25) -- (265.36,327.55) ;
\draw [color={rgb, 255:red, 0; green, 0; blue, 200 }  ,draw opacity=1 ][line width=1.5]    (460.74,270.61) .. controls (496.75,241.98) and (478.25,225.94) .. (526.46,172.24) ;
\draw [color={rgb, 255:red, 0; green, 0; blue, 200 }  ,draw opacity=1 ][line width=1.5]    (526.46,172.24) .. controls (572.26,127.89) and (538.83,117.46) .. (587.04,63.76) ;
\draw [color={rgb, 255:red, 0; green, 0; blue, 200 }  ,draw opacity=1 ][line width=1.5]    (407.65,270.62) .. controls (368.49,233.08) and (393,226.24) .. (348.28,171.07) ;
\draw [color={rgb, 255:red, 0; green, 0; blue, 200 }  ,draw opacity=1 ][line width=1.5]    (348.28,171.07) .. controls (305.38,125.09) and (339.39,117.26) .. (294.68,62.09) ;

\draw (596.96,333.35) node [anchor=north west][inner sep=0.75pt]    {$u$};
\draw (379.11,334.75) node [anchor=north west][inner sep=0.75pt]    {$u^{-}_{M}( x_{2})$};
\draw (442.88,333.75) node [anchor=north west][inner sep=0.75pt]    {$u^{+}_{M}( x_{2})$};
\draw (239.7,333.92) node [anchor=north west][inner sep=0.75pt]    {$u^{+}_{M}( x_{1})$};
\draw (171.84,333.92) node [anchor=north west][inner sep=0.75pt]    {$u^{-}_{M}( x_{1})$};
\draw (200.3,162.49) node [anchor=north west][inner sep=0.75pt]    {$A( x_{1} ,\cdot )$};
\draw (410.56,165) node [anchor=north west][inner sep=0.75pt]    {$A( x_{2} ,\cdot )$};

\end{tikzpicture}

\caption{This illustrates a flux $A(x,u)$ at two spatial points $x_1,x_2\in\R$. Here $A(x_i,u),\,i=1,2$ are functions having flat regions in $[u_M^-(x_1),u_M^+(x_2)],\,i=1,2$  respectively. The mapping $u \mapsto A(x,u)$ attains the same minimum value, independent of $x$.
\label{figure:two_fluxes}}
\end{figure}
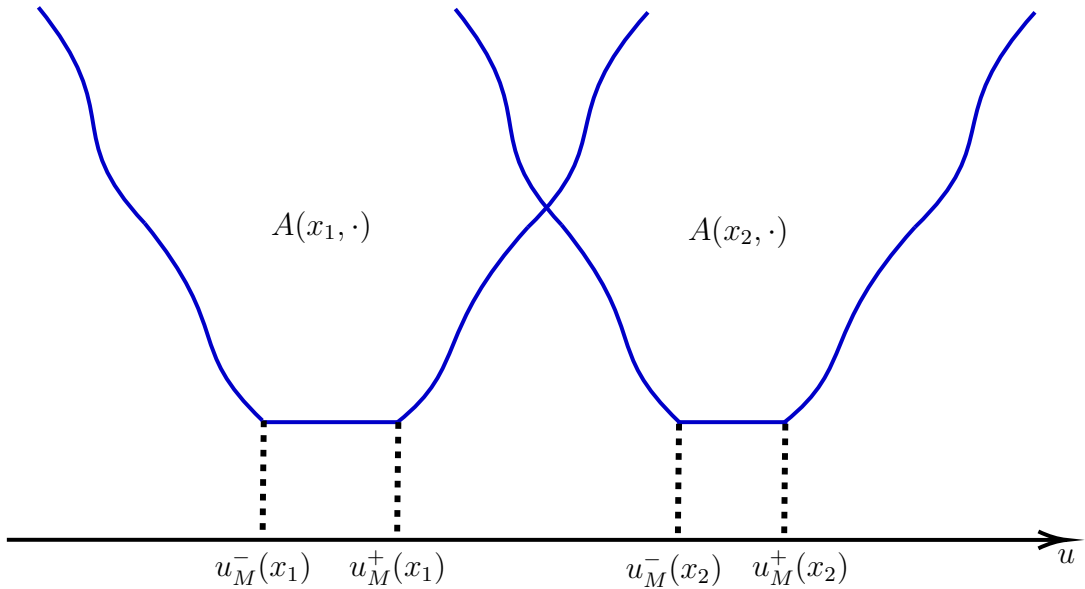

Equations of the type \eqref{eq:discont}--\eqref{eq:data} with spatial discontinuities are well known in mathematical and engineering literature due to their wide range of applications. Some of the well-known applications are sedimentation \cite{diehl1996conservation}, petroleum industry and polymer flooding  \cite{shen2017uniqueness}, two phase flow in heterogeneous porous medium \cite{andreianov2013vanishing}, clarifier thickener unit used in waste water treatment plants \cite{burger2006extended}, traffic flow with abruptly changing road condition  \cite{BKT_2009}, and the hydrodynamic limit of interacting particle systems with discontinuous speed parameter \cite{CEK}. 
The case of a degenerate flux occurs in one version of the  fundamental diagram often used in the transportation engineering literature 
in the so-called Cell Transmission Model (CTM), e.g., \cite{daganzo_94}.
Another case of a degenerate flux corresponds a so-called diphasic behavior  \cite{BCR}. The intervals of state space 
$(-\infty,u_M^-(x))$ 
and $(u_M^+(x),\infty)$ correspond to different phases, and the interval
$(u_M^-,u_M^+)$ (the flat region) corresponds to a mixture of the two phases.

This problem has been gaining recent mathematical interest, especially from the point of  well-posedness. It is well known that when $x \mapsto A(x,u)$ is not sufficiently smooth, the classical Kruzkhov inequality, 
\begin{eqnarray}\label{kruzkov}
{\partial_t} |u-k|+{\partial_x} \left[\sgn(u-k) (A(x,u)-A(x,k))\right] + \sgn(u-k) {\partial_x}A(x,k) \leq 0, \quad k\in \R,
\end{eqnarray}
does not make sense due to the term $\sgn(u-k){\partial_x}A(x,k).$ When $A(x,u)$ has finitely many discontinuities this obstacle is usually overcome by imposing an interface entropy condition at each point of spatial flux discontinuity. As a result, the entropy theories for uniqueness of solutions of conservation laws, with $A(x,u)$ having finitely many discontinuities, generally require existence of traces, to specify the interface entropy condition, see for example, \cite{AJG, AMV,  andreianov2013vanishing, AKR} and the references therein. But for the cases where $A(x,u)$ has infinitely many discontinuities, the existence of traces is not guaranteed. Thus in order to formulate
a definition of solution in this case, an entropy condition that is independent of traces is required.  One of the major developments in this direction is the notion of \textit{adapted entropy solutions} which generalizes the classical Kruzhkov theory to a certain class of fluxes having possibly infinite discontinuities. This theory has the advantage that it does not require the existence of
traces at the location of the spatial flux discontinuities. In this direction,
an adapted entropy inequality was proposed in \cite{AudussePerthame} for monotone and unimodal flux functions by introducing a certain class of steady state solutions in \eqref{kruzkov}.  Solutions satisfying this inequality were shown to be unique. This work was generalized in \cite{Panov2009a} to other class of functions which are of the form $A(x,u)=g(\beta(x,u)),$ where $g \in C(\R)$ and $\beta(x, \cdot)$ is a  monotone function.  Though \cite{Panov2009a} generalizes the notion adapted entropy solution to fairly large class of fluxes, 
it requires that $u \mapsto A(x,u)$ are of same type for all $x$ in some sense. For the fluxes which are unimodal away from the degeneracies, this notion of adapted entropy was extended to obtain the uniqueness of solutions in \cite{GTV_2020}.    The question of existence of solutions satisfying the adapted entropy inequality, was resolved for the monotone fluxes by Piccoli et al. in \cite{piccoli2018general} via wave front tracking, assuming that the fluxes are convex. Numerically, the existence of the solutions with monotone fluxes, without the assumption of convexity of flux $A(x,u)$ was established in \cite {JDT_2020}  by proposing a convergent numerical scheme. This study was further augmented by Ghoshal et al. in \cite{GJT_2019}, where the existence of the solutions was studied via convergence of a Godunov type numerical scheme, for the unimodal (non-monotone) fluxes without any degeneracies. When $A(x,\cdot)$ admits degeneracies, well-posedness was settled in \cite{GTV_2020}, where the existence of the solutions was proved via front tracking algorithm and uniqueness was proven by suitably modifying adapted entropy condition, to circumvent the presence of the degeneracies.

This article focuses on the existence of the adapted entropy solution in  the case of degenerate fluxes, via a convergent Godunov type numerical scheme. Compactness has always been an issue when it comes to convergence of numerical schemes for conservation laws with discontinuous flux, due to the blow up of total variation in finite time. A well known method to tackle this issue is the use of singular maps. However, the classical technique of singular maps cannot be applied in this context as singular map are not invertible when the fluxes are degenerate. In \cite{GTV_2020}, this obstacle of non invertibility of the singular maps was overcome by using the augmented mapping $u\mapsto \Psi(x,u) + \pi(x,u)$, which is invertible and existence was shown via wave front tracking algorithms.

Another important question regarding conservation law with discontinuous flux is the spatial total variation of the solutions. Though the solution operator forms an $L^1$ contractive semigroup, $L^1$ contractivity does not imply that solutions are TVD unlike in the homogeneous case i.e. $A(x,u)=f(u)$ because solutions do not satisfy the translation invariance  property when the fluxes heterogeneous. Thus in general BV bounds on the solutions do not exist.
For the case of a single flux discontinuity, existence and non existence of TV bounds are studied in  \cite{ADGG, Ghoshal-JDE, NHM} using the explicit Lax-Oleinik formulae for conservation laws with discontinuous flux derived in \cite{adimurthi2000conservation}.  These results are not true in general when the fluxes admit infinitely many spatial discontinuities as pointed out in \cite{GTV_2020}. Nevertheless,  existence of BV  bounds were proved  via front tracking in \cite{GTV_2020} for certain class of initial data under the assumption that the fluxes are uniformly convex. Since degenerate fluxes are not uniformly convex, the question remains as to whether  there exists sufficient conditions which assure the existence of TV bounds for the solutions when initial data is in BV and fluxes are allowed to be degenerate.  The BV estimate obtained in this article indeed gives an affirmative answer to this question, for a certain class of degenerate fluxes $A(\cdot,\cdot).$

In the current article we establish existence of the solution via convergence of a Godunov type scheme and show that limit of the finite volume approximation satisfies the adapted entropy inequality obtained in \cite{GTV_2020}. Convergence is proved for two class of flux functions. For the first class, convergence is established by proving BV bounds away from the points of spatial discontinuity and for the second class by proving BV bounds on $\beta(\cdot,u(\cdot,t))$. In the recent years, there has been a considerable developments in the study of BV regularity of conservation laws with discontinuous flux. However, most of the  BV results either assume uniform convexity or monotonicity of the flux function $A(x, \cdot)$. To the best of our knowledge this is the first BV result for conservation laws with discontinuous flux,  which requires neither monotonicity nor uniform convexity of $u \mapsto A(x,u)$. 

As in \cite{GJT_2019}, the scheme and results of
 the present paper can be viewed an extension of those of \cite{AJG}. By employing the adapted entropy approach, we are able to dispense with the regularity assumption appearing in  \cite{AJG}, as well as the restriction to finitely many flux discontinuities. The scheme of \cite{AJG} uses a nonstandard spatial grid that is suitable when there is a single flux discontinuity, but becomes complicated in the case of multiple flux discontinuities. Our scheme employs a
 standard spatial grid. Moreover, with our algorithm
 it is not necessary to locate, identify, or process the flux discontinuities in any special way.
We simply apply the Godunov interface flux at every grid cell boundary. At cell boundaries where there is no flux discontinuity, the interface flux automatically reverts to the classical Godunov flux, as desired. 

\subsection{Adapted Entropy Solutions}
We make the following assumptions about the flux $A(x,u)$:
\begin{enumerate}[label=\textbf{A-\arabic*}]
	\item \label{A1}$A(x,u)$ is continuous on $\mathbb{R}\setminus {\Omega} \times \mathbb{R},$ where $\Omega$ is a closed zero measure set. 
	
	\item \label{A2} There exists a locally bounded function $q:\R\rr\R$ such that
	\begin{equation}
	\abs{A(x,u)-A(x,v)}\leq q(M)\abs{u-v}\mbox{ for a.e. }x\in\R \mbox{ and }u,v \in[-M,M]\mbox{ with }M>0.
	\end{equation}

	\item \label{A3}There exist  functions $u_M^{\pm}:\mathbb{R} \rightarrow \mathbb{R}$ which are continuous on $\R \setminus \Omega,$ such that $u_M^-(x)\leq u_M^+(x)$ for $x\in\R\setminus\Omega$ and $A(x,\cdot)$ is decreasing on $(-\infty,u_M^-(x)]$ and increasing on $[u_M^+(x),\infty)$ satisfying $A(x,z)=0$ for all $u_M^-(x) \leq  z \leq u_M^+(x)$.
    \item \label{A4} We assume that there is a continuous function $\gamma: [0,\infty) \rightarrow [0,\infty)$,
	which is strictly increasing with $\gamma(0)=0$, $\gamma(+\infty) = +\infty$, and
	such that
	\begin{equation}\label{uniform_unimodal}
	\begin{split}
	&\textrm{$A(x,u) \ge \gamma(u-u^+_M(x))$ for all $x\in \R$ and $u \in [u^+_M(x),\infty)$},\\
	&\textrm{$A(x,u) \ge \gamma(-(u-u^-_M(x)))$ for all $x\in \R$ and $u \in (-\infty,u^-_M(x)]$}.
	\end{split}
	\end{equation}
\end{enumerate}

\begin{definition}\label{def:states}
	A function $k: \R \rightarrow \R$ is said to be a stationary state if $u(t,x)=k(x)$ is the weak solution to the IVP \eqref{eq:discont}--\eqref{eq:data}, with $u_0(x)=k(x)$. For $\al>0$, we work with two types of stationary states $k_\al^{+}:\R\rr(u_M^+,\f)$ and $k_\al^{-}:\R\rr(-\f,u_M^-)$ such that 
\begin{equation*}
	A(x,k^{\pm}_\al(x))=\al.
\end{equation*}
We define $\mathscr{S}_\al$ to be the set of all stationary states corresponding to height $\al\geq0$.
\end{definition}
\begin{remark}
	Note that for $\al=0$ there are infinitely many choices for stationary states $k(x)$. A stationary state only needs to satisfy $k(x)\in[u_M^-(x),u_M^+(x)]$. We observe that
	
	$u_M^{\pm}$ can be written as 
	\begin{equation}
	u^-_M(x)=\inf\{u\in \R;\, A(x,u)=0\} \mbox{ and } u^+_M(x)=\sup\{u\in \R;\, A(x,u)=0\}.
	\end{equation}
	For notational brevity we denote a stationary state by $k_\al(x)$ for $\al\geq0$. When $\al>0$, $k_\al$ coincides with one of $k_\al^{\pm}$.
	
\end{remark} 

\begin{definition}[Adapted Entropy Condition] \label{def_adapted_entropy}
	A function $u\in L^{\infty}(Q) \cap C([0,T],L_{loc}^1(\R))$ is an adapted entropy solution of the Cauchy problem  if it satisfies the following inequality in the sense of distributions:
	\begin{equation}\label{ineq:adapted}
	{\partial_t} |u(x,t)-k^{\pm}_{\alpha}(x)| +{\partial_x}\left[ \sgn (u-k^{\pm}_{\alpha}(x)) (A(u,x)-\alpha) \right] \leq 0,
	\end{equation}for $\alpha \geq 0.$
	Or equivalently, for all $0\leq\phi \in C_c^{\infty}([0,T) \times \R^+)$
		\begin{align}
	&\int\limits_Q |u(t,x)-k_{\alpha}(x)|\phi_t(t,x)+\sgn (u(t,x)-k_{\alpha}(x)) (A(x,u(t,x))-\alpha)\phi_x(t,x)\, dx dt\nonumber \\
	&+\int\limits_{\R}|u_0(x)-k_{\alpha}(x)|\phi(0,x)\,dx\geq 0.\label{E2}
	\end{align}

\end{definition}
Before we conclude the discussion on adapted entropy, we state the uniqueness and stability result of the adapted entropy solution. 
\begin{theorem}\label{theorem1}(Uniqueness Theorem \cite{GTV_2020})
	Let $u,v \in  L^{\infty}(Q) \cap C(0,T;L^1_{loc}(\R))$ be entropy solutions to the IVP \eqref{eq:discont}--\eqref{eq:data} with 	initial data $u_0,v_0 \in L^{\infty}(\R).$ Assume the flux satisfies the hypothesis \eqref{A1}--\eqref{A3}. Then for $t\in [0,T]$ 	the following holds,
	\begin{equation*}
	\int\limits_{a}^{b}|u(t,x)-v(t,x)|dx \leq \int\limits_{a+Mt}^{b-Mt}|u_0(x)-v_0(x)|dx,
	\end{equation*}
	for  $-\f\leq a<b\leq \f$ and $M:=\sup\{\abs{A_u(x,u(t,x))};\,x\in\R,0\leq t\leq T\}$.
\end{theorem}

In Section~\ref{sec_godunov} we describe the Godunov scheme of \cite{GJT_2019} and 
extend the convergence result of that paper to the more general setting of this paper, where the flux may be degenerate.
In Section~\ref{sec_BV} we focus on the case where the flux has the so-called Panov form: $A(x,u)=g(\beta(x,u))$.
We prove that the Godunov scheme converges in this case also, and also establish a spatial total variation bound
for the approximate solutions. In Section~\ref{sec_numerical} we present the results of numerical experiments.

  \section{The Godunov scheme and proof of convergence}\label{sec_godunov} 
   
 In this section we show that the Godunov scheme of \cite{GJT_2019} converges to the unique entropy solution.
This provides another existence result, in addition to the one established in \cite{GTV_2020} via
front tracking.
 We state some additional hypotheses, and
then use the compactness method of \cite{BGKT} (the so-called $\BV_{\textrm{loc}}$ method), rather than the singular mapping method that was used in 
\cite{GJT_2019}. Most of the 
relevant lemmas of \cite{GJT_2019} still apply with the setup of this paper. 

We use the compactness method of \cite{BGKT} because $u\mapsto \Psi(x,u)$ is not invertible, due
to the fact that $u \mapsto A(x,u)$ is constant on $[u^-_{M(x)},u^+_{M(x)}]$. 
In \cite{GTV_2020}
we overcame this obstacle by using the augmented mapping $u\mapsto \Psi(x,u) + \pi(x,u)$, which is invertible. We have not discovered how to apply this technique to our Godunov scheme.

In this section, in addition to \eqref{A1}--\eqref{A4}, we assume the following:

\begin{enumerate}[label=\textbf{B-\arabic*}]

\item \label{B1} Referring to Assumption \eqref{A1}, we assume that
the mapping $x \mapsto A(x,u)$ is not only piecewise continuous, but also piecewise constant. We still assume that 
the set of discontinuities $\Omega$ is a closed set of zero
measure.
	
    \item \label{B2} There exists a continuous function $\eta: \mathbb{R} \rightarrow \mathbb{R}$ and a BV function $a:\R \rightarrow \R$ such that \begin{eqnarray}
    |A(x,u)-A(y,u)|\leq \eta(u)|a(x)-a(y)|.
    \end{eqnarray}
    \item \label{B3} $u_M^{\pm} \in \BV(\R)$.
\end{enumerate}
\medskip

Define $u_M(x) := \left(u_M^-(x)+u_M^+(x) \right)/2$. Note that $u_M \in \BV(\R)$, a consequence of   \descref{B3}{B-3}.
As in \cite{GJT_2019}, for now we assume that {\color{blue}$u_0-u_M$} has compact support and $u_0 \in \BV(\R)$.
We will show  that the solution we obtain as a limit of numerical approximations satisfies the 
adapted entropy inequality \eqref{ineq:adapted}. Using Theorem~\ref{theorem1}, the resulting existence theorem is then extended to
the case of $u_0 \in  L^{\infty}(\R)$ via approximations to $u_0$ that are in $\BV$ and
are equal to $u_M$ outside of compact sets.

For $\D x>0$ and $\D t>0$ consider equidistant spatial grid points $x_j:=j\D x$ for $j\in\Z$ and temporal grid points $t^n:=n\D t$ 
for integers $0 \le n\le N$. Here $N$ is the integer such that $T \in [t^N,t^{N+1})$. Let $\la:=\D t/\D x$. Let $\chi_j(x)$ denote the indicator function of $I_j:=[x_j - \D x /2, x_j + \D x /2)$, and let
$\chi^n(t)$ denote the indicator function of $I^n:=[t^n,t^{n+1})$. We approximate the initial data according to:
\begin{equation}
u^{\D}_0:=\sumj\chi_j(x)u^0_j\quad \mbox{where }u^0_j=u_0(x_j)\mbox{ for }j\in\Z.
\end{equation}
The  approximations generated by the scheme are denoted by $u_j^n$, where $u_j^n \approx u(x_j,t^n)$.
The grid function $\{u_j^n\}$ is extended to a function defined on $\Pi_T$ via
\begin{equation}\label{def_u_De}
u^{\D}(x,t) =\sum_{n=0}^N \sumj \chi_j(x) \chi^n(t) u_j^n.
\end{equation}
We use the symbols $\Delta_{\pm}$ to denote spatial difference operators:
\begin{equation}
\Delta_+ z_j = z_{j+1}-z_j, \quad \Delta_- z_j = z_{j}-z_{j-1}.
\end{equation}
We use the same Godunov type scheme that we employed in \cite{GJT_2019}:
\begin{equation}\label{scheme_A}
u_j^{n+1} = u_j^n - \lambda \D_- \bar{A}(u^n_j,u^n_{j+1},x_j,x_{j+1}), \quad j \in \Z, \quad n=0,1,\ldots,N,
\end{equation}
where the numerical flux $\bar{A}$ is the generalized Godunov flux of \cite{AJG}:
\begin{equation}\label{def_bar_A_direct}
\bar{A}(u,v,x_j,x_{j+1}) :=
\max \left\{A(x_j,\max(u,u_M(x_j))) , A(x_{j+1},\min(v,u_M(x_{j+1})))  \right\}.
\end{equation}
$\bar{A}$ is a generalization of the classical Godunov numerical flux \cite{CranMaj:Monoton,leveque_book} in the sense
that
\begin{equation}
\bar{A}(u,v,x,x) = 
\begin{cases}
\min_{w \in [u,v]} A(x,w), \quad & u\le v,\\
\max_{w \in [u,v]} A(x,w), \quad & u\ge v.
\end{cases} 
\end{equation}

Let
\begin{equation}\label{A_2}
\bar{\alpha} = \sup_{x \in \R} A(x,u_0(x)),
\end{equation}
which is finite  due to Assumptions  \descref{A2}{A-2}, \descref{B2}{B-2}, and $u_0 \in \BV(\R)$.
Define $k_{\bar{\alpha}}^{\pm}(x)$ via the equations
\begin{equation}\label{define_k}
\begin{split}
&A(k_{\bar{\alpha}}^{-}(x),x) = \bar{\alpha}, \quad k_{\bar{\alpha}}^{-}(x) \le u^-_M(x),\\
&A(k_{\bar{\alpha}}^{+}(x),x) = \bar{\alpha}, \quad k_{\bar{\alpha}}^{+}(x) \ge u^+_M(x).
\end{split}
\end{equation}
By slightly modifying the proof of  Lemma 3.1 of \cite{GJT_2019} we obtain the following lemma:
\begin{lemma}\label{k_bounded}
	The following bounds are satisfied: 
	\begin{equation}
	\sup_{x \in \R}k_{\bar{\alpha}}^{\pm}(x) < \infty.
	\end{equation}
\end{lemma}
Let 
\begin{eqnarray}\label{def_B_L1}
	\begin{array}{lll}
	\mM = \max\left(\sup_{x \in \R} \abs{k_{\bar{\alpha}}^{-}(x)}, 
	\sup_{x \in \R} \abs{k_{\bar{\alpha}}^{+}(x)}\right), \quad \bar{\eta}=\sup\{\eta(u);\,\abs{u}\leq \mM\},\\
	L =\sup\{\abs{\partial_u A(x,u)}:\abs{u} \le \mM, x \in \R \}.
	\end{array}
\end{eqnarray}
Note that by Assumption \descref{B2}{B-2}, $L< \infty$. Since $\eta$ is continuous we have $\bar{\eta}<\f$.
Also, by \eqref{define_k} we have 
\begin{equation}\label{u_M_bounds}
\textrm{$k^-_{\bar{\alpha}}(x) \le u_M(x)^- \le u_M(x) \le u_M(x)^+ \le k^+_{\bar{\alpha}}(x)$ for all $x \in \R$},
\end{equation}
implying that $\norm{u_M}_{\infty},  \norm{u_M^{\pm}}_{\infty} \le \mM$.
For the convergence analysis that follows we assume that $\D :=(\Delta x,\Delta t)\rightarrow 0$ with the
	ratio $\lambda = \D t / \D x$ fixed and satisfying the CFL condition
	\begin{equation}\label{CFL}
	\lambda L \le 1/2.
	\end{equation}

Lemmas \ref{lemma:k_stationary_A}, \ref{lemma_A_scheme_monotone}, \ref{lemma:A_u_bound} and
\ref{lemma_time_continuity} below are, respectively, Lemmas 3.4, 3.5, 3.6 and 3.8 of \cite{GJT_2019}. The proofs
appearing in \cite{GJT_2019} apply equally well here.
\begin{lemma}\label{lemma:k_stationary_A}
	The grid functions $\{ k_{\bar{\alpha}}^{-}(x_j)\}_{j \in \Z}$ 
	and $\{ k_{\bar{\alpha}}^{+}(x_j)\}_{j \in \Z}$
	are stationary solutions of the difference scheme.
\end{lemma}
\begin{lemma}\label{lemma_A_scheme_monotone}
	The scheme is monotone, meaning that
	if $ \abs{v_j^n}, \abs{w_j^n} \le \mM$ for $ j \in \Z$, then
	\begin{equation*}
	v_j^n \le w_j^n, \quad j \in \Z
	\implies
	v_j^{n+1} \le w_j^{n+1},  \quad j \in \Z.
	\end{equation*}
\end{lemma}
\begin{lemma}\label{lemma:A_u_bound}
                             The Godunov approximations are bounded:
	\begin{equation}\label{u_bounded}
	\abs{u_j^n} \le \mM, \quad j \in \Z, n \ge 0.
	\end{equation}
\end{lemma}
\begin{lemma}\label{lemma_time_continuity}
	The following time continuity estimate holds for $u_j^n$:
	\begin{equation}\label{time_continuity_est_u}
	\sumj \abs{u_j^{n+1}-u_j^n} \le 2 \lambda (\bar{\eta} \TV(a) +L \TV(u_0) + L \TV(u_M)).
	\end{equation}
\end{lemma}

\begin{lemma}\label{lemma_incremental}
	Suppose that the mapping $x\mapsto A(x,u)$ is constant on $[x_{j-1},x_{j+1}]$, i.e.,
	$[x_{j-1},x_{j+1}] \cap \Omega = \emptyset$. Then the Godunov scheme can be
	written in incremental form:
	\begin{equation}\label{inc_form}
	u_j^{n+1} = u_j^n + C_{\jph}^n \D_+ u_j^n - D_{\jmh}^n \D_- u_j^n,
	\end{equation}
	where 
	\begin{equation}\label{coeff_bounds}
	C_{\jph}^n, D_{\jmh}^n \in [0,1], \quad C_{\jph}^n+ D_{\jph}^n \le 1.
	\end{equation}
\end{lemma}

\begin{proof}
It is readily verified that the marching formula \eqref{scheme_A} can be put in the form \eqref{inc_form} where
\begin{equation}
\begin{split}
&C_{\jph}^n = 
\begin{cases}
- \lambda {\bar{A}(u_j^n,u_{j+1}^n,x_j,x_j) - \bar{A}(u_j^n,u_{j}^n,x_j,x_j)  \over  u_{j+1}^n - u_{j}^n}, \quad & u_{j+1}^n - u_{j}^n \neq 0,\\
0, \quad & u_{j+1}^n - u_{j}^n = 0,
\end{cases}\\
&D_{\jmh}^n =
\begin{cases}
\lambda {\bar{A}(u_j^n,u_{j}^n,x_j,x_j) - \bar{A}(u_{j-1}^n,u_{j}^n,x_j,x_j)  \over  u_{j}^n - u_{j-1}^n}, \quad & u_{j}^n - u_{j-1}^n \neq 0,\\
0, \quad & u_{j}^n - u_{j-1}^n = 0.
\end{cases}
\end{split}
\end{equation}
Referring to \eqref{def_bar_A_direct}, it is clear that $C_{\jph}^n, D_{\jmh}^n \ge 0$. From the CFL condition
\eqref{CFL}, we have $\abs{C_{\jph}^n}, \abs{D_{\jmh}^n} \le 1/2$. Thus \eqref{coeff_bounds} holds.
\end{proof}

\begin{lemma}\label{lemma_compactness}
	Fix $T>0$. The Godunov approximations $u^{\D}$ converge (along a subsequence) 
	in $L^1(Q)$ and boundedly a.e. in $Q$ to some 
	$u \in L^{\infty}(Q) \cap C([0,T]:L^1_{\loc}(\R))$.
\end{lemma}

\begin{proof}
Since $\R \setminus \Omega$ is open, it is the union of a countable set of disjoint open intervals,  
$\R \setminus \Omega = \cup_{m=1}^{\infty} (a_m,b_m)$. First consider the spatial interval $(a_1,b_1)$.
Due to Lemmas \ref{lemma:A_u_bound}, \ref{lemma_time_continuity} and \ref{lemma_incremental}, we can repeat the proof
of the compactness portion of Theorem 4.2 of \cite{BGKT}. From this we conclude that  
$u^{\D}$ converges (along a subsequence) in $L^1((0,T) \times (a_1,b_1))$ and boundedly
a.e. in $(0,T) \times (a_1,b_1)$. We can extract a further subsequence of this first subsequence that converges
 in $L^1((0,T) \times (a_2,b_2))$ and boundedly
a.e. in $(0,T) \times (a_2,b_2)$. Applying the Cantor diagonal process, we repeat this for $m=3, \ldots$, 
and then extract the subsequence
along the diagonal, which converges in  $(0,T) \times \cup (a_m,b_m)$. Since $\R \setminus \Omega$ has measure zero, the resulting subsequence converges
in $L^1(Q)$ and boundedly a.e. in $Q$.
The assertion that $u \in L^{\infty}(Q) \cap C([0,T]:L^1_{\loc}(\R))$ follows from Lemmas \ref{lemma:A_u_bound}
and \ref{lemma_time_continuity}.
\end{proof}
\noindent The following is Lemma 4.1 of \cite{GJT_2019}, which remains valid here.
\begin{lemma}\lab{entropy_lemma1_A}
	We have the following discrete entropy inequalities:
	\begin{equation}\label{ent_discrete_A}
	\abs{u^{n+1}_j- k^{\pm}_{\alpha,j}} 
	\le \abs{u_j^n - k^{\pm}_{\alpha,j}}
	- \lambda (\mF^n_{\jph} - \mF^n_{\jmh}),\mbox{ for all }j\in\Z,
	\end{equation}
	where
	\begin{equation*}
	\mathcal{F}^n_{\jph} = \bar{A}(u_j^n \vee k^{\pm}_{\alpha,j},u_{j+1}^n \vee k^{\pm}_{\alpha,j+1},x_j,x_{j+1})	
	                                 -  \bar{A}(u_j^n \wedge k^{\pm}_{\alpha,j},u_{j+1}^n \wedge k^{\pm}_{\alpha,j+1},x_{\Delta x},x_{j+1}).
	\end{equation*}
\end{lemma}

	\noindent The following is (part of) Lemma 3.14 of \cite{GJT_2019}, which remains valid for the setup of this paper.
	\begin{lemma}\label{lemma_a_converge}
	Define
	$a^{\D}(x):=\sumj\chi_j(x)a(x_j)$.
	As $\D \rightarrow 0$, $a^{\D} \rightarrow a$ in $L^1_{\loc}(\R)$ and pointwise a.e.
\end{lemma}

\begin{lemma}\label{lemma_k_converge}
	Define $k_{\alpha,j}^{\pm} = k_{\alpha}^{\pm}(x_j)$, and let
	\begin{equation}\label{k_alpha}
	k_{\alpha}^{\pm,\D}(x) = \sumj \chi_j(x) k_{\alpha,j}^{\pm}.
	\end{equation}
	Then
	\begin{equation}
	\textrm{$k_{\alpha}^{\pm,\D}(x) \rightarrow k_{\alpha}^{\pm}(x)$
	in $L^1_{\loc}(\R)$ and pointwise a.e.}
	\end{equation}
\end{lemma}

\begin{proof}
We prove the assertion for $k_{\alpha}^{+,\D}(x)$. The proof for $k_{\alpha}^{-,\D}(x)$ is
similar. Define 
\begin{equation}
\textrm{$A^+(x,\cdot): (u_M^+(x),\infty)\mapsto (0,\infty)$ by
$A^+(x,u) = A(x,u)$ with $u \in (u_M^+(x), \infty)$. }
\end{equation}
Note that for each $x\in \R$, 
$A^+(x,\cdot)$ has a continuous, single-valued inverse, $(A^+)^{-1}(x,\cdot)$, which is
defined by
\begin{equation}\label{A_inv}
(A^+)^{-1}(x,\alpha) = k^+_{\alpha}(x), \quad \alpha >0.
\end{equation}

We apply $A^+(x,\cdot)$ to both sides of the ``$+$'' version of \eqref{k_alpha}, resulting in
\begin{equation}\label{A_diff}
\begin{split}
	A^+(x,k_{\alpha}^{+,\D}(x)) 
	&= \sumj \chi_j(x) A^+(x,k_{\alpha,j}^{+})\\
	&= \sumj \chi_j(x) A^+(x_j,k_{\alpha,j}^{+}) 
	  + \sumj \chi_j(x)\left( A^+(x,k_{\alpha,j}^{+}) - A^+(x_j,k_{\alpha,j}^{+})\right)\\
	&= \alpha
	  + \sumj \chi_j(x)\left( A^+(x,k_{\alpha,j}^{+}) - A^+(x_j,k_{\alpha,j}^{+})\right). 
\end{split}
\end{equation}
From \eqref{A_diff} and the estimate
\begin{equation}
\abs{A^+(x,k_{\alpha,j}^{+}) - A^+(x_j,k_{\alpha,j}^{+})} \le \bar{\eta}\abs{a(x)-a(x_j)},
\end{equation}
we find that
\begin{equation}\label{A_alpha_1}
\begin{split}
	\abs{A^+(x,k_{\alpha}^{+,\D}(x)) - \alpha} 
	&\le \bar{\eta}\sumj \chi_j(x) \abs{a(x) - a(x_j)}\\
	&= \bar{\eta} \abs{a(x) - a^{\D}(x)}.
\end{split}
\end{equation}
The right hand side of \eqref{A_alpha_1} converges to zero in $L^1_{\textrm{loc}}(\R)$ as $\D \rightarrow 0$, 
thanks to Lemma~\ref{lemma_a_converge}.
Thus $A^+(x,k_{\alpha}^{+,\D}(x)) \rightarrow \alpha$ in $L^1_{\textrm{loc}}(\R)$.
Since $A^+(x,k_{\alpha}^{+,\D}(x)) > 0$, $\alpha>0$, we also have
$\abs{A^+(x,k_{\alpha}^{+,\D}(x))} \rightarrow \abs{\alpha}$ in $L^1_{\textrm{loc}}(\R)$.
This implies that  $A^+(x,k_{\alpha}^{+,\D}(x)) \rightarrow \alpha$ pointwise a.e.
Next, we invoke the continuity of $(A^+)^{-1}(x,\cdot)$ and employ \eqref{A_inv} to conclude that
$k_{\alpha}^{+,\D}(x) \rightarrow k_{\alpha}(x)$ a.e., and thus also in
$L^1_{\textrm{loc}}(\R)$.
\end{proof}

\noindent The following is Lemma 4.3 of \cite{GJT_2019}. The proof appearing in \cite{GJT_2019} applies equally well here.
\begin{lemma}\label{lemma_lim_kruzkov}
	The (subsequential) limit $u$ guaranteed by Lemma~\ref{lemma_compactness} satisfies 
	the adapted entropy inequalities 
	of Definition~\ref{def_adapted_entropy}.
\end{lemma}
\noindent The following is basically Theorem 2.5 of \cite{GJT_2019}, whose proof is also valid here.
\begin{theorem}\label{thm_convergence}
	Assume that the flux function $A$ satisfies Assumptions  \descref{A1}{A-1} through \descref{A4}{A-4},
	and
	Assumptions  \descref{B1}{B-1}
	and \descref{B2}{B-2}. Also assume that $u_0\in L^{\infty}(\R)$. 
	Then as the mesh size $\D \rightarrow 0$, the approximations $u^{\D}$ generated by the Godunov scheme described 
	above converge in $L^1_{\loc}(Q)$ and pointwise a.e. in $Q$ to the unique adapted entropy solution 
	$u \in L^{\infty}(Q) \cap C([0,T]:L^1_{\loc}(\R))$ corresponding to the Cauchy problem 
	\eqref{eq:discont}, 
	\eqref{eq:data} with initial data $u_0$.
\end{theorem}
 \section{The case of a Panov-type flux}\label{sec_BV}

For the case where $A(x,u)$ is unimodal ($u_M^-(x)=u_M^+(x) = u_M(x)$), Panov \cite{Panov2009a}
observed that the flux can be written in the form
\begin{equation}\label{panov_form}
A(x,u) = g(\beta(x,u)),
\end{equation}
 where $u\mapsto \beta(x,u)$ is strictly increasing, and $g$ is continuous and unimodal. Based on this observation,
 well-posedness (in the sense of Audusse-Perthame entropy solutions) was established in \cite{Panov2009a}
 (for the unimodal case).
 
 In this section, we assume that the flux has the form \eqref{panov_form}, but we allow for the
 the possibility that  $g$ may be degenerate, i.e., not unimodal.
It turns out that due to the special form \eqref{panov_form} it is possible to obtain a 
$\BV$ bound for the Godunov approximations (and thus the limit), whether or not the flux
is degenerate. At
the same time, this approach makes it possible to obtain compactness/existence
without the piecewise constant assumption of the previous section (the first part of Assumption  \descref{B1}{B-1}). We also do not need
the assumption that $u_M^{\pm} \in \BV(\R)$ (Assumption  \descref{B3}{B-3}).

One does not generally expect to obtain $\BV$ bounds for solutions to
conservation laws with discontinuous flux, so our $\BV$ bound may
seem a little surprising. However, our results are consistent with
previous of results of \cite{Ghoshal-JDE,NHM} where it was noted that 
the total variation is bounded if the flux has a minimum value that
is independent of $x$, which is the case here.

The specific assumptions about $g$ and $\beta$ are:
\begin{enumerate}[label=\textbf{C-\arabic*}]
\item \label{C1} 
The flux has the form $A(x,u) = g(\beta(x,u))$, where the properties of $g$ and $\beta$
are described below.

\item \label{C2}
For some $z^-, z^+$ with  $z^- \le 0 \le z^+$,
$g$ is strictly decreasing on $(-\infty,z^-)$ and strictly increasing on $(z^+,\infty)$, and 
$g(z) = 0$ for $z \in [z^-,z^+]$. 
	There is a continuous function $\kappa: [0,\infty) \rightarrow [0,\infty)$,
	which is strictly increasing with $\kappa(0)=0$, $\kappa(+\infty) = +\infty$, and
	such that
	\begin{equation}\label{uniform_unimodal_kappa}
	\begin{split}
	&\textrm{$g(z) \ge \kappa(z-z^+)$ for all $z \in [z^+,\infty)$},\\
	&\textrm{$g(z) \ge \kappa(-(z-z^-))$ for all $z \in (-\infty,z^-]$}.
	\end{split}
	\end{equation}

\item \label{C3}
$g(z)$ is (locally) Lipschitz-continuous, i.e.,
\begin{equation}\label{g_lip}
\textrm{$\abs{g(z_1) - g(z_2)} \le \mK_1(r) \abs{z_1-z_2}$ for $z_1, z_2 \in [-r,r]$},
\end{equation}
where $\mK_1:\R \rightarrow [0,\infty)$ is continuous.

\item \label{C4}
$\beta(x,u)$ is continuous on $\mathbb{R}\setminus {\Omega} \times \mathbb{R},$ 
where $\Omega$ is a closed zero measure set.
In addition, $u \mapsto \beta(x,u)$ is strictly increasing, and for each $x \in \R$,
$\abs{\beta(x,u)} \rightarrow \infty$
as $\abs{u} \rightarrow \infty$.

\item \label{C5}
For $u,v \in [-r,r]$
\begin{equation}\label{beta_2}
\abs{\beta(x,v)-\beta(x,u)} \le \mK_3(r) \abs{u-v},
\end{equation}
for some continuous $\mK_3: \R \rightarrow [0,\infty)$.
Also,
\begin{equation}\label{beta_1A}
\abs{\beta(x,u)-\beta(y,u)} \le \mK_4(u) \abs{\alpha(x)-\alpha(y)},
\end{equation}
where $\mK_4: \R \rightarrow [0,\infty)$ is continuous and $\alpha \in \BV(\R)$.

\item \label{C6} For some $\mK_2 >0$, independent of $x$
\begin{equation}\label{beta_1}
\abs{\beta(x,u)-\beta(x,v)} \ge \mK_2 \abs{u-v}.
\end{equation}

\end{enumerate}

\begin{remark}\normalfont
With this setup, the degeneracy comes from Assumption \descref{C2}{C-2}, 
from which it is clear that  $g'(z)=0$ on the interval $[z^-,z^+]$. 
\end{remark}

\begin{remark}\label{remark_inverse}\normalfont
Note that
since $u \mapsto \beta(x,u)$ is strictly increasing, there is an inverse, denoted
$\beta^{-1}(x,u)$. Moreover, $u \mapsto \beta^{-1}(x,u)$ is continuous for $x \in \R \setminus \Omega$.
Define
\begin{equation}
u_M^-(x)=\beta^{-1}(x,z^-), \quad u_M^+(x)=\beta^{-1}(x,z^+), \quad u_M(x) = \beta^{-1}(x,0).
\end{equation}
Due to the monotonicity of $u \mapsto \beta^{-1}(x,u)$, we have
  $u_M^-(x) \le u_M(x) \le u_M^+(x)$. Moreover,  $A(x,u) = 0$ for $u \in [u_M^-(x),u_M^+(x)]$, and
$u \mapsto A(x,u)$ is strictly decreasing on $(-\infty,u_M^-(x)]$
and strictly increasing on $[u_M^+(x),\infty)$.
\end{remark}

As mentioned above, the idea to take $A(x,u) = g(\beta(x,u))$ comes from
Panov's \cite{Panov2009a}. Panov proved existence and
uniqueness of this problem (under somewhat different 
regularity assumptions) using different analytical methods.
To our knowledge, the total variation bounds that we 
derive below are new.

\begin{example}
Suppose that for some $z^- \le 0 \le z^+$,
\begin{equation}
g(z) = 
\begin{cases}
(z-z^-)^2, &\quad z< z^-,\\
0, &\quad  z^- \le z \le z^+,\\
(z-z^+)^2, &\quad z>z^+.\\
\end{cases}
\end{equation}
If $\beta(x,u) = u-r(x)$, then 
$A(x,u) = g(u-r(x))$. All of the required hypotheses are satisfied if 
$r \in \BV(\R)$ and the set of its discontinuities is closed.
If $\beta(x,u) = s(x)u$, where $s(x) \ge \underbar{s}>0$, then 
$A(x,u) = g(s(x)u)$. If $z^-=z^+=0$ (the nondegenerate case), all of the required hypotheses are satisfied if 
$s \in \BV(\R)$ and the set of its discontinuities is closed.
\end{example}

\begin{example}
Given a conservation law in ``capacity form'' \cite[Section 2.4]{leveque_book}:
\begin{equation*}
\theta(x) v_t + f(v)_x = 0, \quad \theta(x) \ge \theta_{\textrm{min}} >0,
\end{equation*}
by making the change of variables $u = \theta(x) v$, we obtain
\begin{equation*}
u_t = f(u/\theta(x)) = 0.
\end{equation*}
In this case, $\beta(x,u) = u/\theta(x)$.
\end{example}

\begin{remark}\normalfont
In the unimodal case ($u_M^-(x) = u_M^+(x)=u_M(x)$), the assumption that the flux has the form
\eqref{panov_form} is not really a restriction. Indeed, reference \cite{Panov2009a} observes
that in the unimodal case the flux $A(x,u)$ can be written in the form \eqref{panov_form} with
\begin{equation}
\beta(x,u) = \sgn(x-u_M(x))A(x,u), \quad g(\beta) = \abs{\beta}.
\end{equation}
\end{remark}

\begin{lemma}\label{lemma_uM_bdd}
The following bound holds for some constant $\mK_0 \in \R$:
\begin{equation}
\textrm{$\abs{u_M(x)} \le \mK_0$ for all $x\in \R.$}
\end{equation}
\end{lemma}

\begin{proof}
Using \descref{C6}{C-6} and $\beta(x,u_M(x)) = 0$, we find that
\begin{equation}
\begin{split}
\abs{\beta(x,u)}
&= \abs{\beta(x,u)-\beta(x,u_M(x))}\\
&\ge \mK_2 \abs{u - u_M(x)}.
\end{split}
\end{equation}
Substituting $u=0$, we obtain
\begin{equation}
\abs{u_M(x)} \le \abs{\beta(x,0)}/\mK_2.
\end{equation}
The proof is completed by observing that $\beta(x,0)$ is bounded, due to
the second part of Assumption \descref{C5}{C-5}.
\end{proof}

\begin{lemma}\label{lemma_Asmptns_A}
Assumptions \descref{A1}{A-1} through \descref{A4}{A-4} hold, as does Assumption \descref{B2}{B-2}.
\end{lemma}
\begin{proof}
Assumption \descref{A1}{A-1} holds, due to the assumptions
that $\beta(x,u)$ is continuous on $\mathbb{R}\setminus {\Omega} \times \mathbb{R}$,
and that $g$ is continuous on $\R$. 

To prove that Assumption \descref{A2}{A-2} holds, assume that $u,v \in [-M,M]$. Then
\begin{equation}\label{bound_beta}
\begin{split}
\abs{\beta(x,u)} 
&= \abs{\beta(x,u) - \beta(x,u_M(x))}\,\, \textrm{since $\beta(x,u_M(x))=0$}\\
&\le \mK_3(M) \abs{u - u_M(x)}\,\, \textrm{by \descref{C5}{C-5}}\\
&\le  \mK_3(M) \left(M + \mK_0\right)\,\, \textrm{by Lemma~\ref{lemma_uM_bdd}}.\\
\end{split}
\end{equation}
Define $\tilde{\mK}_3(M):= \mK_3(M) \left(M + \mK_0\right)$.
We have $\abs{\beta(x,u)} \le \tilde{\mK}_3(M)$, and similarly, $\abs{\beta(x,v)} \le  \tilde{\mK}_3(M)$.
Thus,
\begin{equation}
\begin{split}
\abs{A(x,u)-A(x,v)} 
&=\abs{g(\beta(x,u)) - g(\beta(x,v))}\\
&\le \mK_1(\tilde{\mK}_3(M))\abs{\beta(x,u)-\beta(x,v)}\,\, \textrm{by \descref{C3}{C-3}}\\
&\le \mK_1(\tilde{\mK}_3(M)) \mK_3(M) \abs{u-v}\,\, \textrm{by \descref{C5}{C-5}}.
\end{split}
\end{equation}
Thus, Assumption \descref{A2}{A-2} holds with $q(M) = \mK_1(\tilde{\mK}_3(M)) \mK_3(M)$.

That  Assumption \descref{A3}{A-3} holds is the content of
Remark~\ref{remark_inverse}.
To verify that Assumption \descref{A4}{A-4} holds, 
recall from Remark~\ref{remark_inverse} that $A(x,u)$ is strictly decreasing on $(-\infty,u_M^-(x)]$ and 
strictly increasing on $[u_M^+(x), \infty)$. In fact by \eqref{uniform_unimodal_kappa}, and since $\abs{\beta(x,u)} \rightarrow \infty$
as $\abs{x} \rightarrow \infty$, we have that $A(x,u) \rightarrow \infty$
as $\abs{x} \rightarrow \infty$.
Define $\gamma(u) = \kappa(\mK_2 u)$.
It is clear that $\kappa(0)=0$, $\kappa(+\infty) = +\infty$.
To show that the first inequality of \eqref{uniform_unimodal} holds, let $u \ge u_M^+(x)$,
\begin{equation}
\begin{split}
A(x,u) 
&= g(\beta(x,u))\\
&\textrm{$\ge \kappa(\beta(x,u) - z^+)$ by \eqref{uniform_unimodal_kappa}}\\
&\textrm{$= \kappa(\beta(x,u) - \beta(x,u_M^+(x)))$ by substituting $\beta(x,u_M^+(x))=z^+$}\\
&\textrm{$\ge \kappa(\mathcal{K}_2 (u-u_M^+(x)))$ by \eqref{beta_1} and monotonicity of $\kappa$}\\
&=\gamma(u-u_M^+(x)).
\end{split}
\end{equation}
The proof of the second inequality of \eqref{uniform_unimodal} is similar.

To prove that Assumption \descref{B2}{B-2} holds we use \eqref{bound_beta} 
 with $M = \abs{u}$, 
which yields
\begin{equation}
\abs{\beta(x,u)}\le \tilde{\mK}_3(\abs{u}).
\end{equation}
Then,
\begin{equation}
\begin{split}
\abs{A(x,u)-A(y,u)}
&=\abs{g(\beta(x,u)) - g(\beta(y,u))}\\
&\le \mK_1(\tilde{\mK}_3(\abs{u}))\abs{\beta(x,u)-\beta(y,u)}\,\, \textrm{by \descref{C3}{C-3}}\\
&\le  \mK_1(\tilde{\mK}_3(\abs{u})) \mK_4(u) \abs{\alpha(x)-\alpha(y)}\,\, \textrm{by \descref{C5}{C-5}}.
\end{split}
\end{equation}
Thus Assumption \descref{B2}{B-2} holds with $\eta(u) = \mK(\tilde{\mK}_3(\abs{u})) \mK_4(u)$, 
and $a(x) = \alpha(x)$.
\end{proof}

As in Section~\ref{sec_godunov}, for now we assume that {\color{blue}$u_0-u_M$} has compact support and $u_0 \in \BV(\R)$.
Due to Lemma~\ref{lemma_Asmptns_A} we can proceed as in Section~\ref{sec_godunov}.
With $\bar{\alpha}$ and $k_{\bar{\alpha}}^{\pm}$ defined by \eqref{A_2} and \eqref{define_k},
Lemmas~\ref{k_bounded} and \ref{lemma:k_stationary_A} are valid here also. We can also define the
quantities $\mathcal{M}, \bar{\eta}, L$ appearing in \eqref{def_B_L1}, and the ordering appearing in
\eqref{u_M_bounds} holds, in particular
 $\norm{u_M}_{\infty} \le \mM$.

\begin{lemma}\label{lemma_beta_bdd}
For each $Z >0$,
\begin{equation}
\sup_{\abs{u} \le Z, x\in \R} \abs{\beta(x,u)} < \infty.
\end{equation}
\end{lemma}

\begin{proof}
We have
\begin{equation}
\begin{split}
\abs{\beta(x,u)} 
&= \abs{\beta(x,u)-\beta(x,u_M(x))}\,\, \textrm{using $\beta(x,u_M(x))=0$}\\
&\le  \mK_3(Z) \abs{u-u_M(x))}\\
&\le \mK_3(Z) (Z + \mM).
\end{split}
\end{equation}
Here we have used $\norm{u_M}_{\infty} \le \mM$.
\end{proof}
Let $\mP = \sup_{\abs{u} \le \mM, x\in \R} \abs{\beta(x,u)}$, and
define
\begin{equation}
L_{\beta} = \mK_3(\mM), \quad L_g = \mK_1(\mP).
\end{equation}
We assume that the following  CFL condition holds:
\begin{equation}\label{CFL_1}
\lambda L_g L_{\beta} \le 1/2.
\end{equation}
It is readily verified that the Lipschitz constant $L$ appearing in the CFL condition \eqref{CFL}
satisfies $L \le L_g L_{\beta}$, and so \eqref{CFL} also holds.

Lemmas \ref{lemma_A_scheme_monotone} (monotonicity) 
and \ref{lemma:A_u_bound} ($u_j^n$ bounded) of Section~\ref{sec_godunov}
are applicable with the more restrictive CFL condition \eqref{CFL_1} in effect. 

For the remainder of the convergence analysis we make use of the fact that
the scheme can be written in terms of the $\beta$ variable.
It is clear from Assumption  \descref{C2}{C-2} that
$g$ is nonincreasing on $(-\infty,0)$ and nondecreasing on $(0,\infty)$. As a consequence
 the Godunov flux $\bar{g}$ 
consistent with $g$ can be expressed as follows:
\begin{equation}\label{bar_g}
\bar{g}(p,q) = \max\{g(\max(p,0)),g(\min(q,0)) \}.
\end{equation}
\begin{lemma}\label{barA_barg}
	The following relationship between the Godunov fluxes $\bar{A}$ and $\bar{g}$ holds:
	\begin{equation}
	\bar{A}(u,v,x,y) = \bar{g}(\beta(x,u), \beta(y,v)).
	\end{equation}
\end{lemma}
\begin{proof}
	We claim that
	\begin{equation}\label{beta_facts}
	\beta(x,\max(u,u_M(x))) = \max(\beta(x,u),0), \quad \beta(y,\min(v,u_M(y))) = \min(\beta(y,v),0).
	\end{equation}
	To verify the first part of the claim, note that since $u \mapsto \beta(x,u)$ is nondecreasing,
	\begin{equation}
	\begin{split}
	 \beta(x,\max(u,u_M(x))) 
	    &= \max(\beta(x,u), \beta(x,u_M(x)))\\
	    &= \textrm{$\max(\beta(x,u), \beta(x,\beta^{-1}(x,0))$), using $u_M(x)=\beta^{-1}(x,0)$}\\
	    &= \max(\beta(x,u),0).
	\end{split}
	\end{equation}
	The second assertion of \eqref{beta_facts} is verified in a similar manner.
	
	Next, starting from \eqref{def_bar_A_direct} and then using $A(x,u) = g(\beta(x,u))$, along with \eqref{beta_facts} and 
	\eqref{bar_g}, we find that
	\begin{equation}\label{scheme}
	\begin{split}
	\bar{A}(u,v,x,y) 
	&= \max \{A(x,\max(u,u_M(x))), A(y,\min(v,u_M(y)))\}\\
	&= \max \{g(\beta(x,\max(u,u_M(x)))), g(\beta(y,\min(v,u_M(y))))\}\\
	&= \max \{g(\max(\beta(x,u),0)), g(\min(\beta(y,v),0)\}\\
	&=\bar{g}(\beta(x,u), \beta(y,v)).
	\end{split}
	\end{equation}
\end{proof}
	Let $\beta_j^n = \beta(x_j,u_j^n)$.
	Lemma~\ref{barA_barg} makes it possible to write the marching formula \eqref{scheme_A} in the
	equivalent form
\begin{equation}\label{scheme_beta}
	u_j^{n+1} = u_j^n - \lambda \D_- \bar{g}(\beta_j^n, \beta_{j+1}^n), \quad j \in \Z, \quad n=0,1,\ldots,N,
\end{equation}
	which we abbreviate as
\begin{equation}\label{scheme_beta_short}
	\textrm{$u_j^{n+1} = u_j^n - \lambda \D_- \bar{g}_{\jph}^n$, 
	where $\bar{g}_{\jph}^n=\bar{g}(\beta_j^n, \beta_{j+1}^n)$}.
\end{equation}

\begin{lemma}\label{lemma_godunov_tvd}
	The Godunov scheme is TVD with respect to $\{\beta_j^n\}$ in the following sense:
    \begin{eqnarray}
    \sum\limits_{j\in \mathbb{Z}}\abs{\beta_{j+1}^{n+1} - \beta_j^{n+1}} \le \sum\limits_{j\in \mathbb{Z}}\abs{\beta_{j+1}^{n} - \beta_j^{n}},
    \end{eqnarray}
    and for some $\Delta$-independent constant $\mK_5>0$,
\begin{equation}\label{bv_beta_n}
\sumj \abs{\beta_{j+1}^{n} - \beta_j^{n}} \le \mK_5.
\end{equation}

\end{lemma}

\begin{proof}
First, by combining Lemmas~\ref{lemma:A_u_bound} and \ref{lemma_beta_bdd},
we have 
\begin{equation}\label{bound_beta_P}
\abs{\beta_j^n}  \le \mP, \quad j \in \Z, \quad n \ge 0.
\end{equation}
Thus $L_g$ serves as a Lipschitz constant for $g(\cdot)$, and also
for both arguments of the Godunov numerical flux $\bar{g}(\cdot,\cdot)$
in the calculations that follow.
Next, we apply
 $\beta(x_j, \cdot)$ to both sides of \eqref{scheme_beta_short}, which
yields
\begin{equation}\label{tv_beta_1}
\beta(x_j,u_j^{n+1}) = \beta(x_j,u_j^n - \lambda \D_- \bar{g}_{\jph}^n).
\end{equation}
The right side of \eqref{tv_beta_1} can be expressed as
\begin{equation}
\beta(x_j,u_j^n - \lambda \D_- \bar{g}_{\jph}^n) =
\beta(x_j,u_j^n) - \lambda \theta_j^{\nph} \D_- \bar{g}_{\jph}^n,
\end{equation}
where, using $u_j^{n+1}-u_j^n = - \lambda \D_- \bar{g}_{\jph}^n$,
\begin{equation}
 \theta_j^{\nph} = 
 \begin{cases}
 {\beta(x_j,u_j^{n+1}) - \beta(x_j,u_j^n) \over u_j^{n+1}-u_j^n},
 \quad & u_j^{n+1}-u_j^n \neq 0,\\
0, \quad & u_j^{n+1}-u_j^n =0.
\end{cases}
\end{equation}
Thus we have
\begin{equation}\label{beta_scheme}
\beta_j^{n+1} = \beta_j^n - \lambda \theta_j^{\nph} \Delta_- \bar{g}^n_{\jph},
\end{equation}
and it is clear that 
\begin{equation}\label{theta_bound}
 0 \le  \theta_j^{\nph} \le L_{\beta}.
\end{equation}
Next we write \eqref{beta_scheme} in incremental form:
\begin{equation}\label{beta_incremental}
\beta_j^{n+1} = \beta_j^n +\mC_{\jph}^n \Delta_+ \beta_j^n - \mD_{\jmh}^n \Delta_- \beta_j^n,
\end{equation}
where
\begin{equation}
\begin{split}
\mC_{\jph}^n &=
\begin{cases}
 - \lambda \theta_j^{\nph} 
\left({\bar{g}(\beta_{j}^n,\beta_{j+1}^n) - \bar{g}(\beta_{j}^n,\beta_j^n)\over \beta_{j+1}^n - \beta_j^n }\right), 
\quad & \beta_{j+1}^n - \beta_j^n \ne 0,\\
0, \quad & \beta_{j+1}^n - \beta_j^n = 0,
\end{cases}\\
\mD_{\jmh}^n &=  
\begin{cases}
\lambda \theta_j^{\nph} 
\left({\bar{g}(\beta_{j}^n,\beta_j^n) - \bar{g}(\beta_{j-1}^n,\beta_j^n)\over \beta_{j}^n - \beta_{j-1}^n }\right),
\quad & \beta_{j}^n - \beta_{j-1}^n \ne 0,\\
0, \quad & \beta_{j}^n - \beta_{j-1}^n = 0.
\end{cases}
\end{split}
\end{equation}
Recalling \eqref{theta_bound} and that
$\bar{g}(\cdot,\cdot)$ is a monotone numerical flux \cite{CranMaj:Monoton,leveque_book}, we have
$\mC_{\jph}^n, \mD_{\jmh}^n \ge 0$. 
Moreover, due to the CFL condition \eqref{CFL_1},
along with the fact that $L_g$ serves as a Lipschitz constant for 
$\bar{g}(\cdot,\cdot)$,
we also have
\begin{equation}\label{inc_coeffs_bound}
\mC_{\jph}^n, \mD_{\jmh}^n \le \lambda L_g L_{\beta} \le 1/2. 
\end{equation}
Thus $\mC_{\jph}^n+ \mD_{\jph}^n \le 1$, 
and we can apply Harten's lemma \cite[Theorem 6.1]{leveque_book}, which yields
\begin{equation}\label{beta_tvd}
\sumj \abs{\beta_{j+1}^{n+1} - \beta_j^{n+1}} \le \sumj \abs{\beta_{j+1}^{n} - \beta_j^{n}},
\end{equation}
and verifies the claim that the scheme is TVD with respect to $\{ \beta_j^n\}$.
It follows from \eqref{beta_tvd} that 
\begin{equation}\label{beta_tv_bound}
\sumj \abs{\beta_{j+1}^{n} - \beta_j^{n}} \le \sumj \abs{\beta_{j+1}^{0} - \beta_j^{0}}.
\end{equation}
We estimate the terms in the sum on the right side of \eqref{beta_tv_bound}:
\begin{equation}\label{u0_diffs}
\begin{split}
\abs{\beta_{j+1}^{0} - \beta_j^{0}} 
&= \abs{\beta(u_{j+1}^0,x_{j+1}) - \beta(u_{j}^0,x_{j})}\\
&\le \abs{\beta(u_{j+1}^0,x_{j}) - \beta(u_{j}^0,x_{j})} + \abs{\beta(u_{j+1}^0,x_{j+1}) - \beta(u_{j+1}^0,x_{j})}\\
&\le \mK_3(x_j) \abs{u_{j+1}^0 - u_j^0} + \mK_4(u_{j+1}^0) \abs{\alpha(x_{j+1})-\alpha(x_j)}.
\end{split}
\end{equation}
In light of \eqref{u0_diffs}, along with the assumption that $u_0$ is compactly supported, $u_0 \in \BV(\R)$, 
and Assumption  \descref{C5}{C-5},
we have 
for some $\mK_5>0$ independent of the mesh size $\D$,
	\begin{equation}\label{bv_beta_0}
	\sumj \abs{\beta_{j+1}^{0} - \beta_j^{0}}  \le \mK_5.
	\end{equation}
Finally, in view of \eqref{beta_tv_bound} and \eqref{bv_beta_0} we obtain the  spatial variation bound \eqref{bv_beta_n}.
\end{proof}

\begin{remark}\normalfont  If $A(x,u)$ is unimodal, then \eqref{C1} is satisfied with $g(u)=\abs{u}$ and $\beta(x,u)=\Psi_A(x,u).$ In \cite{GJT_2019}, convergence was established by showing that  $\Psi_A(\cdot,u^{\Delta x}(\cdot,t))$ is TVB. The above lemma is stronger in the sense that,  it implies $\Psi_A(\cdot,u^{\Delta x}(\cdot,t))$ is not only TVB but also in addition TVD.
\end{remark}

\begin{lemma}\label{lemma:TV_u}
For some $\Delta$-independent constant $\mK_6>0$,
\begin{equation}\label{bv_u_n}
\sumj \abs{u_{j+1}^{n} - u_j^{n}} \le \mK_6.
\end{equation}
\end{lemma}

\begin{proof}
We employ the reverse triangle inequality to obtain
\begin{equation}\label{beta_diff}
\begin{split}
\abs{\beta_{j+1}^{n} - \beta_j^{n}}
& \ge \abs{\beta(x_j,u_{j+1}^n)- \beta(x_j,u_j^n)} - \abs{\beta(x_{j+1},u_{j+1}^n)- \beta(x_j,u_{j+1}^n)}\\
& \ge \mK_2 \abs{u_{j+1}^n-u_j^n} - \abs{\beta(x_{j+1},u_{j+1}^n)- \beta(x_j,u_{j+1}^n)}.
\end{split}
\end{equation}
Here we have used Assumption  \descref{C6}{C-6}. From \eqref{beta_diff} we have
\begin{equation}
\abs{u_{j+1}^n-u_j^n} \le {1 \over \mK_2} \left(\abs{\beta_{j+1}^{n} - \beta_j^{n}} +
    \abs{\beta(x_{j+1},u_{j+1}^n)- \beta(x_j,u_{j+1}^n)}  \right).
\end{equation}
The proof is completed by summing over $j \in \Z$ and invoking \eqref{bv_beta_n}
and Assumption  \descref{C5}{C-5}.
\end{proof}

\begin{lemma}\label{godunov_time_cont}
We have the following discrete time continuity estimates:
\begin{eqnarray}\label{time_cont}
\sumj \abs{\beta_j^{n+1}-\beta_j^n} \le \mK_5,\\
\sumj \abs{u_j^{n+1}-u_j^n} \le \mK_7,
\end{eqnarray}
where $\mK_7>0$ is independent of the mesh size $\D$.
\end{lemma}

\begin{proof}
From \eqref{beta_incremental} and \eqref{inc_coeffs_bound}, we have
\begin{equation}
\begin{split}
\sumj \abs{\beta_j^{n+1} - \beta_j^n} 
&\le \sumj \left(\mC_{\jph}^n\abs{\D_+ \beta_j^n} + \mD_{\jmh}^n\abs{\D_- \beta_j^n}\right)\\
&\le \sumj \left({1 \over 2} \abs{\D_+ \beta_j^n} + {1 \over 2} \abs{\D_- \beta_j^n}\right)
= \sumj \abs{\D_+ \beta_j^n} \le\mK_5 .
\end{split}
\end{equation}
Here we have used \eqref{bv_beta_n} to obtain the last inequality above.
Using the first inequality of \eqref{time_cont}, along with \eqref{beta_1}, we find that
\begin{equation}
\sumj \abs{u_j^{n+1} - u_j^n} \le  \mK_5 / \mK_2=:\mK_7.
\end{equation}
\end{proof}

\begin{theorem}\label{theorem2}
	Assume that the flux function $A(x,u)$ satisfies   Assumptions \descref{C1}{C-1} through \descref{C6}{C-6}, 
	and that $u_0\in L^{\infty}(\R)$. 
	Then as the mesh size $\D \rightarrow 0$, the approximations $u^{\D}$ generated by the Godunov scheme described 
	above converge in $L^1_{\loc}(Q)$ and pointwise a.e. in $Q$ to the unique adapted entropy solution 
	$u \in L^{\infty}(Q) \cap C([0,T]:L^1_{\loc}(\R))$ corresponding to the Cauchy problem \eqref{eq:discont}, 
	\eqref{eq:data} with initial data $u_0$. In addition, the total variation $u(\cdot,t)$ is uniformly  bounded for $t\geq 0$.
\end{theorem}

\begin{proof}
From the $L^{\infty}$ bound (Lemma~\ref{lemma:A_u_bound}), 
the spatial variation bound on $\{u_j^n \}$ (Lemma~\ref{lemma:TV_u}), and the time continuity estimate
(Lemma~\ref{godunov_time_cont}), we have convergence of the
approximations $u^{\D}$ along a subsequence in $L^1_{\textrm{loc}}(Q)$ and boundedly a.e. to 
some $u \in L^{\infty}(Q) \cap C([0,T]:L^1_{\loc}(\R))$.

Lemmas~\ref{entropy_lemma1_A}, \ref{lemma_a_converge}, \ref{lemma_k_converge}, \ref{lemma_lim_kruzkov} 
are valid with the setup of this section, from which we conclude that
the subsequential limit $u$  satisfies 
the adapted entropy inequalities 
of Definition~\ref{def_adapted_entropy}.

Since Assumptions \descref{A1}{A-1} through \descref{A3}{A-3}
are satisfied, we can invoke Theorem~\ref{theorem1}, thus concluding that
the entire sequence $u^{\D}$ (not just a subsequence) converges to the unique entropy solution.

By Lemma~\ref{lemma:TV_u},
we have a spatial variation bound on $u^{\D}(\cdot,t)$ which is independent of the mesh size, i.e.,
	for some $\mK_6>0$ independent of the mesh size $\D$,
	\begin{equation}\label{u_bv_bound}
	\TV(u^{\D}(\cdot,t))  \le \mK_6.
	\end{equation}
	Since $\TV(u(\cdot,t)) \le \liminf \TV(u^{\D}(\cdot,t))$, we also have $\TV(u(\cdot,t)) \le \mK_6.$
\end{proof}

The following examples  illustrates the applications and optimality of the conditions assumed in the Lemma \eqref{lemma_godunov_tvd}.
\begin{example}\normalfont
Consider the flux $A(x,u) = H(-x)u^2+H(x)\abs{u}$ which is of the form $A(x,u)=g(\beta(x,u)),$ where $g(u)=\abs{u}$ and $\beta(x,u)=H(-x)[u^2H(u)-u^2H(-u)]+H(x)u$ which violates Assumption \descref{C6}{C-6} .   There exists $u_0 \in \BV(\R)$ as constructed in \cite{ADGG} such that total variation of the solution in fact blows up in finite time. 
\end{example}
\begin{example}\normalfont
 Suppose $A(x,u)=(u-r(x))^2,$ with $r\in \BV(\R).$    Clearly $A(x, \cdot)$ is uniformly convex for all $x\in \R.$ Theorem 4.2 in \cite{GTV_2020} guarantees the existence of BV bound for initial datum in certain sub-class of BV functions whereas Theorem \ref{theorem2} implies that solutions are of bounded variation whenever $u_0 \in \BV(\R)$.
 \end{example}
 \begin{example}\normalfont
 $A(x,u)=H(-x)g(2u+sin(u)) +H(x)g(2u+cos(u)),$ where\begin{equation}g(u) = 
\begin{cases}
-u-1, \quad & u <-1,\\
0, \quad & u \in (-1,0)\\
u, \quad & u>1,\\
\end{cases} 
\end{equation}
The results in \cite{ADGG, GTV_2020, Ghoshal-JDE, NHM} are not applicable. On the other hand, Theorem \ref{theorem2} indeed assures the existence of total variation bounds on the solutions corresponding to the initial data $u_0\in \BV$.  
 \end{example}

\section{Numerical Simulations}\label{sec_numerical}

This section displays the  performance of the numerical scheme. We consider two examples. In each example the  flux has infinitely many discontinuities with an accumulation point.  Assumptions \eqref{C1}-\eqref{C6} are satisfied for each flux. Numerical experiments are performed on the space interval $[0,6]$ with $M=100, 400$ and $ 800$ spatial grid points. At any given time $t$ such that for $n\in \mathbb{N},$ $n\Delta t \in [t- \Delta t,t), $ let
\[e_{\Delta x}=\Delta x\sum_i\abs{u_i^n-u(x_i,t)}\]
and
\[\TV(\beta(\cdot,u^{\Delta x}(\cdot,t)))=\sum_i\abs{\beta(x_{i},u_i^n)-\beta(x_{i-1},u_{i-1}^n)}, \TV(u^{\Delta x}(\cdot,t))=\sum_i\abs{u_i^n-u_{i-1}^n},\]
denote the approximate $L^1$ error of numerical approximation $u^{\Delta x}(\cdot,t)$ with respect to the exact solution $u(\cdot,t)$,  the total variation of $u^{\Delta x}(\cdot,t)$ and  the total variation of $\beta(\cdot,u^{\Delta x}(\cdot,t))$ respectively.
\begin{example}\label{num_ex_1} \normalfont 
This example demonstrates that the scheme captures solutions containing both rarefactions and shocks efficiently. We consider the flux of the of the form $A(x,u)=g(\beta(x,u))$ with $g(u)=u^2/2, \beta(x,u)=u+r(x)$. Here $r$ is a function of bounded variation and is chosen such that the resulting solution is piecewise linear with infinitely many discontinuities. In fact, the solution is a combination of  infinitely many shocks and rarefactions.
Let $p=4$ and $q=0.8$. Define $C_n=[a_n,a_{n+1}]$ for $n\in \mathbb{N}$ with
\begin{eqnarray}
a_1=1 \text{ and }a_n=1+\sum_{i=1}^{n-1} \tilde{a}_i \text{ for }n\geq 2
\end{eqnarray}
such that for each $n\in \mathbb{N}$,
\begin{equation*}
\tilde{a}_n= 
\begin{cases}
pq^{n-1}-pq^n, \quad & \text{ if } n \text{ is odd},\\
pq^{n-2}-pq^{n-1}, \quad & \text{ if } n \text{ is even}.
\end{cases}
\end{equation*}
Define 
\begin{equation*}
r(x)= 
\begin{cases}
p, \quad & x <1,\\
pq^{n-1}, \quad & x \in C_n,n\in \mathbb{N},\\
 0, \quad & x>a_{\infty},
\end{cases} 
\end{equation*} and consider a piecewise constant initial data 
\begin{equation}
u_0(x)= 
\begin{cases}
-pq, \quad & x <a_2,\\
-pq^n, \quad & x \in C_n \text{ and } n \text{ odd},\\
-pq^{n-2}, \quad & x \in C_n \text{ and } n \text{ even},\\
0, \quad & x>a_{\infty}.
\end{cases} 
\end{equation}At $t=1,$ the solution is a combination of rarefactions and stationary shocks along the spatial discontinuities of $A(\cdot,\cdot)$ and is given by, 
\begin{equation}\label{exact}
u(x,1)= 
\begin{cases}
-pq, \quad & x <a_2,\\
x-a_{n}-pq^{n-1}, \quad & x \in C_{n} \text{ and } n \text{ odd},\\
x-a_{n+1}-pq^{n-1}, \quad & x \in C_{n} \text{ and } n \text{ even},\\
0, \quad & x>a_{\infty}.
\end{cases} 
\end{equation}
Figure \ref{fig:ex4} plots the numerical solutions at the final time $t=1$ for various mesh sizes. It can be seen that the scheme captures both stationary shocks and rarefactions efficiently and that the difference between the approximation and the exact solution decreases as the mesh size reduces. 
\begin{figure}[H]
 \centering
 \includegraphics[width=\textwidth,keepaspectratio]{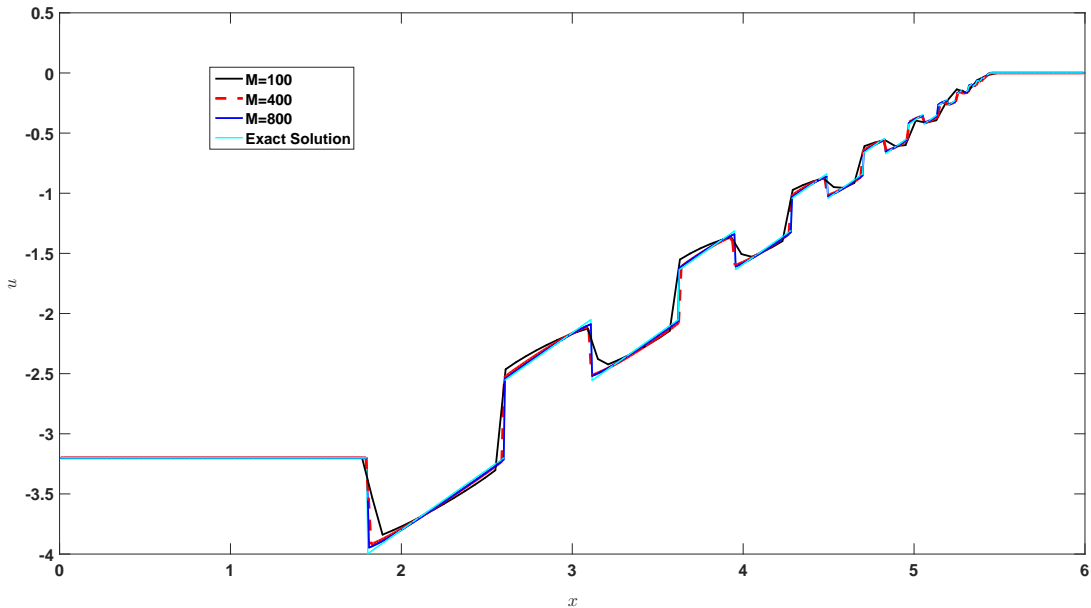}
   \caption{Example~\ref{num_ex_1}. The solution at $t=1$ contains both shocks and rarefactions. The approximation improves
 with decreasing mesh size.}
 \label{fig:ex4}
\end{figure}
 Table \ref{table:1} records $e_{\Delta x}$, $ \TV(u^{\Delta x}(\cdot,1))$, and $\TV(\beta(\cdot,u^{\Delta x}(\cdot,1)))$ for various mesh sizes 
 $\Delta x$. As indicated by Figure \ref{fig:ex4}, Table \ref{table:1}  also reflects that $e_{\Delta x}$ decreases with decreasing mesh size indicating the convergence of the scheme.
 \begin{table}[H]
     \centering
    \begin{tabular}{ |c|c|c|c|c|c| } 
     \hline
 M&  $e_{\Delta x}$ &$\TV(u^{\Delta x}(\cdot,1))$&$\TV(\beta(\cdot,u^{\Delta x}(\cdot,1))$\\\hline
 50& 0.2244 & 5.4883&5.6646\\\hline
 100&  0.1603 &5.7466& 5.9861\\\hline
 200&0.1047&6.2989 &6.6729\\\hline
 400&0.0673&6.5784& 6.9716\\\hline
 800& 0.0423&6.8433 &7.2655\\\hline
 1600&0.0258&7.0394 &7.4769\\\hline
     \end{tabular}
     \caption{ Approximate $L^1$ error and total variation at $t=1$ for Example \ref{num_ex_1}.}
   \label{table:1}
 \end{table}

 \end{example}
\begin{example}\label{num_ex_2}\normalfont 
For this example the flux is 
$A(x,u)=g(\beta(x,u))=g(u-r(x))$ where
\begin{equation}
g(u) = 
\begin{cases}
-u-1, \quad & u <-1,\\
0, \quad & u \in (-1,0),\\
u, \quad & u>1.\\
\end{cases} 
\end{equation}
The initial data is constant: $u_0(x) = 2$, and the resulting solution consists of constant segments separated by shocks.
We approximate the solution numerically up to the final time $t=6$ using the numerical scheme \eqref{scheme_A}. 
For every $n \in \mathbb{N},$ consider the sequence defined by,\begin{eqnarray}
a_n&=&1+0.8+0.8^2+...+0.8^n=5(1-0.8^n), \nonumber\\
r_n&=&1-(-0.8)^n.\nonumber
\end{eqnarray} Now, the function $r$ is defined as follows:
\begin{equation}
r(x)= 
\begin{cases}
2, \quad & x <1,\\
r_n \chi_{[a_n,a_{n+1}]}(x), \quad & x \in (1,5),\\
1, \quad & x>5.\\
\end{cases} 
\end{equation}

The flux considered here admits infinitely many spatial discontinuities with an accumulation point at $x=5.$ 
The initial data is constant but the solution develops discontinuities (immediately) for $t>0$, and the
solution at $t=6$ is given by
\begin{equation}
u(x,6)= r(x) \quad \text{ for } x\in [0,6].
\end{equation}
Note that $\beta(x,u(x,6))=u(x,6)-r(x)=0$ for $x\in [0,6],$
implying that $u(x,6)$ is a stationary solution of the conservation law.
\begin{figure}[H]
 \centering
 \includegraphics[width=\textwidth,keepaspectratio]{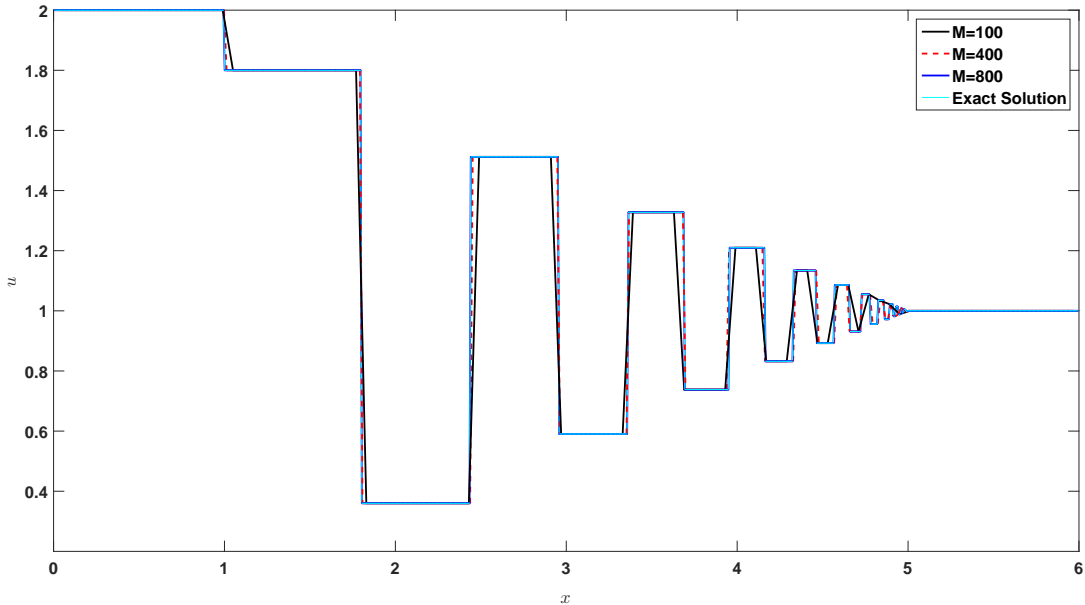}
 \caption{Example~\ref{num_ex_2}. The solution at $t=6$ contains infinitely many shocks along the spatial discontinuities of the flux, with an accumulation point at $x = 5.$ 
 The approximation improves with decreasing mesh size.}
\label{fol}
\end{figure}

Table \ref{table2} records $e_{\Delta x}$, $\TV(u^{\Delta x}(\cdot,6))$, and $\TV(\beta(\cdot,u^{\Delta x}(\cdot,6)))$ for various mesh sizes 
$\Delta x$. Figure \ref{fol} plots the numerical approximation for varying mesh sizes. Table \ref{table2} and Figure \ref{fol} show that the numerical scheme is able to capture the shocks efficiently and as the mesh size goes to zero, and that the computed solution matches well with the exact solution.  
The higher than normal rate of convergence displayed in Table \ref{table2} is due to the fact that $u(\cdot,6)$ is a stationary solution
of the conservation law. At early times, the convergence rate is more like what is seen in Table \ref{table:1}.
Additionally, in Table \ref{table3} we display the total variation at various times for a fixed mesh size, where it can be seen that $\TV(u(\cdot,t)$ remains bounded, but does not decrease, whereas $\TV(\beta(\cdot,u(\cdot,t)))$ in fact  decreases with increase in $t.$
  \begin{table}[H] 
     \centering
    \begin{tabular}{ |c|c|c|c|c|c| } 
     \hline
 M&  $e_{\Delta x}$& $\TV(u^{\Delta x}(\cdot,6))$&$\TV(\beta(\cdot,u^{\Delta x}(\cdot,6))$\\\hline
 50&4.65542e-03& 6.7835& 1.4533e-02\\\hline
 100& 4.2652e-04& 6.9856& 2.2975e-03\\\hline
 200&1.8578e-05& 7.2348 &  1.7906e-04\\\hline
 400&5.8576e-08&7.2804&   1.0639e-06\\\hline
 800& 2.0119e-10&7.3572 &7.3026e-11
\\\hline
 1600&1.2861e-12& 7.3807&   2.0350e-13\\\hline
     \end{tabular}
     \caption{ Approximate $L^1$ error and total variation at $t=6$ for Example \ref{num_ex_2}. The very
     rapid rate of convergence is due to the fact that $u(\cdot,6)$ is a stationary solution.}
  \label{table2}
 \end{table}

\begin{table}[H]
     \centering
    \begin{tabular}{ |c|c|c|c|c|c| } 
     \hline
 t&  $\TV(u^{\Delta x}(\cdot,t))$&$\TV(\beta(\cdot,u^{\Delta x}(\cdot,t))$\\\hline
 0&    0 & 7.2804\\\hline
  1&  12.1201 & 5.2839\\\hline
  2& 11.1309& 4.2867 \\\hline
  3&9.8179&2.9384 \\\hline
  4&8.6358& 1.4036\\\hline
  5&7.3744&0.0941 \\\hline
  6&   7.2804& 1.0639e-06\\\hline
 \end{tabular}
      \caption{Comparison of total variation for various time steps with M=400 spatial grid points, for Example \ref{num_ex_2}.
      $\TV(u^{\Delta x}(\cdot,t))$ increases between $t=0$ and $t=1$, while $\TV(\beta(\cdot,u^{\Delta x}(\cdot,t))$
      decreases over each time interval.}
      \label{table3}
    \end{table}
\end{example}
 
 \noindent\textbf{Acknowledgement.} SSG and GV would like to acknowledge the Department of Atomic Energy, Government of India, under project no. 12-R\&D-TFR-5.01-0520 for the support. SSG thanks the Inspire faculty-research grant DST/INSPIRE/04/2016/000237.


\section{References}
 \begin{thebibliography}{99}
 	\bibitem{AJG}
	\newblock Adimurthi, J. Jaffr\'e and G. D. Veerappa Gowda, 
	\newblock Godunov-type methods for conservation laws with a flux function discontinuous in space,
	\newblock{\em SIAM J. Numer. Anal.} 42 (2004), no. 1, 179--208.
 	
 	\bibitem{ADGG}
 	\newblock  Adimurthi, R. Dutta, S. S. Ghoshal and G. D. Veerappa Gowda. 
 	\newblock Existence and nonexistence of TV bounds for scalar conservation laws with discontinuous flux, 
 	\newblock {\em Comm. Pure Appl. Math.} 64 (2011), no. 1, 84--115.
 	
 	
 	 	
	\bibitem{AMV}
 	\newblock Adimurthi, S. Mishra and G. D. Veerappa Gowda, 
 	\newblock Optimal entropy solutions for conservation laws with discontinuous flux functions,
 	\newblock{\em J. Hyperbolic Differ. Equ.} 2 (2005), 783--837.

 	\bibitem{adimurthi2000conservation}
 	\newblock Adimurthi and G. D. Veerappa Gowda,  
 	\newblock Conservation law with discontinuous flux,
 	\newblock {\em J. Math. Kyoto Univ.}, 2000.
 	
 		
 	\bibitem{andreianov2013vanishing}
 	\newblock B. Andreianov and C. Canc\`{e}s,  
 	\newblock Vanishing capillarity solutions of buckley–leverett equation with gravity in two-rocks medium,
 	\newblock {\em Computational Geosciences} 17(3) (2013), 551--572.
 	
 	\bibitem{AKR}
 	\newblock B. Andreianov, K. H. Karlsen and N. H. Risebro,
 	\newblock A theory of $L^1$-dissipative solvers for scalar conservation laws with discontinuous flux,
 	\newblock{\em Arch. Ration. Mech. Anal.} 201, 1 (2011), 27--86.
 	
 	
 	\bibitem{AudussePerthame} 
 	\newblock  E. Audusse and B. Perthame,  
 	\newblock Uniqueness  for scalar conservation laws with discontinuous flux via adapted entropies,  
 	\newblock{\em Proc.\  Roy.\ Soc.\ Edinburgh Sect. A} 135 (2005), 253--265. 

	\bibitem{BCR}
	\newblock B.~Boutin, C.~Chalons, and P.~Raviart, 
	\newblock Existence result for the coupling problem of two scalar conservation laws with Riemann initial data,
	 {\em Math. Models Methods Appl. Sci.} 20 (2010), 1859--1898.
	
 	\bibitem{BGKT}
 	\newblock R. B\"urger, A. Garc{\'\i}a, K. Karlsen and J. D. Towers,
 	\newblock A family of numerical schemes for kinematic flows with discontinuous flux, 
 	\newblock {\em J. Eng. Math.} 60(3-4) (2008), 387--425.
 	
 	\bibitem{burger2006extended}
 	\newblock R. B\"urger, A. Garcia, K. H. Karlsen, and J. D. Towers,  
 	\newblock On an extended clarifier-thickener model  with  singular  source  and  sink  terms,
 	\newblock {\em European  Journal  of  Applied  Mathematics} 42817(3), (2006), 257--292.

   	 \bibitem{BKT_2009}
     	\newblock R.~B\"{u}rger, K. H.~Karlsen and J.D.~Towers, 
  	\newblock A conservation law with discontinuous flux modelling traffic flow with abruptly changing road surface conditions,
  	\newblock {\em Hyperbolic problems: theory, numerics and applications}, vol. 67 (2009), 455--464.
 	
	\bibitem{CEK}
	\newblock G. Q. Chen, N. Even and C. Klingenberg,
	 \newblock Hyperbolic conservation laws with discontinuous fluxes and hydrodynamic limit for particle systems,
	 \newblock {\em J. Differ. Equ.} 245(11), (2008), 3095--3126. 
 	
 	\bibitem{CranMaj:Monoton} 
 	\newblock M. G.~Crandall and A.~Majda,	
 	\newblock Monotone difference approximations for scalar conservation laws, 
 	\newblock {\em Math.\ Comp.}  34 (1980), 1--21.
 	
	\bibitem{daganzo_94}
	\newblock C. F. Daganzo, 
	\newblock The cell transmission model: A dynamic representation of highway traffic consistent with the hydrodynamic 	theory,
	\newblock {\em Transp. Res. Part B Methodol.} 28(4) (1994), 269--287.
	
 	 \bibitem{diehl1996conservation}
 	\newblock S.~ Diehl,
 	\newblock A conservation law with point source and discontinuous flux function modelling continuous sedimentation,
 	\newblock {\em SIAM Journal on Applied Mathematics} 56(1980), 388--419.

 	\bibitem{GJT_2019}
 	\newblock S. S. Ghoshal, A.~Jana and J. D. Towers,
 	\newblock Convergence of a Godunov scheme to an Audusse-Perthame 
 	adapted entropy solution for conservation laws with BV spatial flux,
 	\newblock {\em Numer. Math.} 146(3) (2020), 629--659.
 	
   	\bibitem{GTV_2020}
 	\newblock S. S. Ghoshal, J. D. Towers and G.~Vaidya,
 	\newblock Well-posedness for conservation laws with spatial heterogeneities and a study of BV regularity,
                             \newblock Preprint https://arxiv.org/pdf/2010.13695.pdf
    
  	 \bibitem{Ghoshal-JDE}
   	\newblock S. S. Ghoshal,
   	\newblock Optimal results on TV bounds for scalar conservation laws with discontinuous flux,
   	\newblock {\em J. Differential Equations} 258 (2015) 980--1014.
   
   	\bibitem{NHM} 
   	\newblock S. S. Ghoshal, 
  	 \newblock BV regularity near the interface for nonuniform convex discontinuous flux, 
   	\newblock {\em Networks and Heterogeneous Media} 11, no.2, (2016), 331--348. 

	\bibitem{leveque_book}
	\newblock R. J. Leveque, 
	\newblock Finite volume methods for hyperbolic problems,
	\newblock Cambridge University Press, Cambridge, UK, 2002.
		
	 \bibitem{Panov2009a}
	\newblock E. Y. Panov,
	\newblock On existence and uniqueness of entropy solutions to the Cauchy
	problem for a conservation law with discontinuous flux,
	\newblock  {\em J. Hyperbolic Differ. Equ.} 06 (2009), 525--548.	
	
 	\bibitem{piccoli2018general}
 	\newblock B. Piccoli and M. Tournus,  
 	\newblock A general BV existence result for conservation laws with spatial heterogeneities,
 	\newblock {\em SIAM Journal on Mathematical Analysis} 50(3) (2018), 2901--2927.
	 	
 	\bibitem{shen2017uniqueness}
 	\newblock W.~ Shen, 
 	\newblock On the uniqueness of vanishing viscosity solutions for riemann problems for polymer flooding,
 	\newblock {\em Nonlinear Differential Equations Appl. NoDEA  24,} 37(2017).

 	\bibitem{JDT_2020} 
 	\newblock  J. D. Towers, 
 	\newblock An existence result for conservation laws having BV spatial flux heterogeneities - without concavity, 
 	\newblock {\em J. Differ. Equ.} 269 (2020), 5754--5764.
 \end{thebibliography}
 \end{document}